\documentclass[12pt,a4paper,dvipsnames]{article}

\usepackage{hyperref}
\usepackage{pgf}
\usepackage[utf8]{inputenc}

\usepackage{mathtools}
\usepackage{amsfonts}
\usepackage{amssymb}
\usepackage{amsthm}
\usepackage{dsfont}
\usepackage[linesnumbered]{algorithm2e}
\usepackage{setspace}
\usepackage{extarrows}

\usepackage{graphicx}
\usepackage{multicol}
\PassOptionsToPackage{dvipsnames}{xcolor}
\usepackage{tikz,xcolor}

\usepackage{pgfplots}

\usetikzlibrary{shapes,arrows}
\usetikzlibrary{hobby}

\usepackage[english]{babel}		%Deutscherechtschreibung
\usepackage[T1]{fontenc}			%Trennungsregeln
\usepackage{color}         			%Graphiken einbinden
\usepackage{listings}					%Code einbinden

\renewcommand{\subsectionmark}[1]{}

\usepackage{geometry}
\newtheorem{theorem}{Theorem}[section]
\newtheorem{lemma}[theorem]{Lemma}
\newtheorem{proposition}[theorem]{Proposition}

\newtheorem{corollary}[theorem]{Corollary}

\newtheorem{remark}[theorem]{Remark}
\newtheorem{question}[theorem]{Question}
\renewcommand{\P}{\mathbb{P}}
\newcommand{\TV}[1]{{\lVert #1 \rVert}_{\normalfont
\text{TV}}}
\newcommand{\N}{\mathbb{N}}
\newcommand{\Z}{\mathbb{Z}}
\newcommand{\R}{\mathbb{R}}
\newcommand{\E}{\mathbb{E}}

\newcommand{\Var}{\textup{Var}}

\title{Cutoff on trees is rare}
\author{Nina Gantert$^{\ast}$, Evita Nestoridi$^{\ast\ast}$ and Dominik Schmid$^{\ast}$}
\date{\today}

\begin{document}

%\setlength{\parindent}{0pt}
%\pagenumbering{gobble}

\vspace{-0.3cm}

\maketitle

\begin{abstract} We study  the simple random walk on trees and give estimates on the mixing and relaxation time. Relying on a recent characterization by Basu, Hermon and Peres, we give geometric criteria, which are easy to verify and allow to determine whether the cutoff phenomenon occurs. We thoroughly discuss families of trees with cutoff, and show how our criteria can be used to prove the absence of cutoff for several classes of trees, including spherically symmetric trees, Galton--Watson trees of a fixed height, and sequences of random trees converging to the Brownian CRT.
\end{abstract}

%\vspace{-0.4cm}

\phantom{.} \hspace{0.2cm}\textbf{Keywords:} Random walk, spectral gap,  mixing times, cutoff phenomenon \\
\phantom{.} \hspace{0.55cm} \textbf{MSC 2020:}  60J27, 60K37 %Primary:Secondary:

\let\thefootnote\relax
\footnotetext{ $^\ast$ \textit{TU Munich, Germany. E-Mail}: \nolinkurl{nina.gantert@tum.de}, \nolinkurl{dominik.schmid@tum.de} \\\phantom{.} \hspace{0.3cm}
$^{\ast\ast}$ \textit{Princeton University, United States. E-Mail}: \nolinkurl{exn@princeton.edu}  }

\section{Introduction and main results}\label{sec:Introduction}

In recent breakthroughs, random walks on various families of random graphs were shown to exhibit the  cutoff phenomenon, which is a sharp transition in the convergence to equilibrium \cite{BS:CutoffNonBack,BLPS:CutoffErdos,LS:CutoffRegular}.
%random walks on random graphs are a topic of huge recent interest. In recent breakthroughs, Lubetzky and Sly proved that the simple random walk on a random $d$--regular graph exhibits the cutoff phenomenon, a sharp transition in the convergence to equilibrium \cite{LS:CutoffRegular}.  Berestycki, Lubetzky, Peres and Sly \cite{BLPS:CutoffErdos} proved the same result for the  Erd\H{o}s--Renyi graph, which was also independently studied by Ben-Hamou and Salez \cite{BS:CutoffNonBack}. 
Motivated by these results, we address the question for which families of (random) trees, the simple random walk has cutoff. In particular, we investigate the simple random walk on Galton--Watson trees, when the offspring distribution $\mu$ has mean $m \geq 1$ and finite variance $\sigma^2 \in (0,\infty)$. Moreover, we study the simple random walk on spherically symmetric trees. Our results on cutoff for these families of trees are summarized as follows.
 \begin{theorem}\label{thm:CutoffExamples} The simple random walk does (almost surely) not exhibit cutoff if the underlying trees $(G_n)_{n\geq 1}$ come from one of the following three constructions:
 \begin{itemize}
 \item  A supercritical Galton--Watson tree ($m>1$) conditioned on survival and truncated at generation $n$.
 \item  A family of  Galton-Watson trees conditioned on having $n$ sites.
 \item  A spherically symmetric tree of bounded degree truncated at generation $n$.
 \end{itemize}
 \end{theorem}
 We refer to Section \ref{sec:CutoffAtypical} for precise formal definitions and statements. Intuitively, Theorem \ref{thm:CutoffExamples} is due to the fact that Galton--Watson trees are typically \textit{short and fat}, as shown by Addario-Berry in \cite{AB:ShortAndFat}, while spherically symmetric trees are \textit{balanced} in the number of sites by construction. We provide criteria showing that trees are typically \textit{long}, \textit{thin} and \textit{imbalanced} when leading to cutoff, see Theorems \ref{thm:NoCutoff} and \ref{thm:CutoffMiclo}. Our arguments rely on a characterization of the cutoff phenomenon on trees that was recently proven by Basu, Hermon and Peres \cite{BHP:CutoffTrees}. 
Moreover, we give an explicit construction for trees such that cutoff occurs, see Theorem \ref{thm:CutoffCompactification}. This allows us to give a simple example of trees for which the simple  random walk exhibits cutoff, see Corollary \ref{cor:NewTreesWithCutoff}. Previously, a first example of trees with cutoff was obtained in \cite{PS:CutoffTree}. \\

We now give a brief introduction to the theory of mixing times and refer to \cite{LPW:markov-mixing} for a more comprehensive treatment.
Let $G=(V,E,o)$ be a finite, rooted graph. Let $(\eta_t)_{t \geq 0}$ denote the (continuous-time) simple random walk on $G$ in the variable speed model, i.e.\ the walker moves along every edge of the tree at rate $1$. For two probability measures $\mu,\nu$ on $V$, let their \textbf{total-variation distance} be
\begin{equation}\label{def:TVdistance}
\TV{\mu-\nu} := \max_{A \subseteq V} \left|\mu(A)-\nu(A)\right| \ .
\end{equation} The $\boldsymbol{\varepsilon}$\textbf{-mixing time} of $(\eta_t)_{t \geq 0}$ is now given, for $\varepsilon \in (0,1)$, by
\begin{equation}\label{def:MixingTime}
t_{\textup{mix}}(\varepsilon) := \inf \left\{ t \geq 0 \colon \max_{x \in V} \TV{\P(\eta_t \in \cdot | \eta_0=x) - \pi }\leq \varepsilon \right\}
\end{equation}
where $\pi$ is the uniform distribution on $V$. For a sequence of graphs $(G_n)_{n \geq 1}$, let $(t^n_{\textup{mix}}(\varepsilon))_{n \geq 1}$ denote the collection of mixing times of the random walks on $(G_n)$. We say that the family of random walks on $(G_n)$ exhibits \textbf{cutoff} if for any $\varepsilon \in (0,1)$
\begin{equation}\label{def:Cutoff}
\lim_{n \rightarrow \infty} \frac{t^n_{\textup{mix}}(1-\varepsilon)}{t^n_{\textup{mix}}(\varepsilon)} = 1 \ .
\end{equation} 
The cutoff phenomenon was first verified in \cite{DS:SpectrumCompleteGraph}, and obtained its name in the seminal paper of Aldous and Diaconis \cite{AD:Shuffling}. Ever since there has been a lot of activity towards showing that specific examples of Markov chains exhibit cutoff, see for example \cite{BS:CutoffNonBack,BLPS:CutoffErdos,LS:CutoffRegular}.  The ultimate goal is to produce a necessary and sufficient condition for cutoff. In turns out that 
cutoff has deep connections to the spectral properties of the underlying random walk, see the results  \cite{EHP:Hardy,LMW:SpectralGapTrees,M:SpectrumMarkovChains,M:EigenfunctionsOnTrees} of Miclo and others for a discussion of the spectrum of random walks on trees, and the work of Jonasson \cite{J:MixingInterchange} for a collection of results on the spectrum of random walks on random graphs. \\

 It was conjectured by Peres that a necessary and sufficient condition for cutoff is that for some $\varepsilon \in (0,1)$ (or, equivalently, for all $\varepsilon \in (0,1)$)
\begin{equation}\label{eq:ProductCriterionIntro}
\lim_{n \rightarrow \infty} t_{\text{mix}}^n(\varepsilon) \lambda_n = \infty  \ ,
\end{equation} where $\lambda_n$ denotes the spectral gap of the random walk on $G_n$, see also Section \ref{sec:PrelimSpectral} for a formal definition. However, while this product criterion is indeed necessary, Aldous showed that it is not sufficient, see Chapter 18 of \cite{LPW:markov-mixing}. Nevertheless, it is believed  to be sufficient for a wide range of families of Markov chains. \\

\vspace{-0.5mm}For birth-and-death chains, it was shown by Ding, Lubetzky and Peres that the product criterion is sharp \cite{DLP:BirthDeath}. Moreover, Chen and Saloff-Coste provide simple conditions such that \eqref{eq:ProductCriterionIntro} holds \cite{CS:SpectrumBirthDeath,CS:CutoffTimes}. Using their results, we will be able to exclude cutoff on spherically symmetric trees in Section \ref{sec:Spherically}, following a projection argument due to Nestoridi and Nguyen for $d$--regular trees \cite{NN:SpectrumTree}. Recently, it was also verified for the simple random walk on trees by Basu, Hermon and Peres  \cite{BHP:CutoffTrees} that the product criterion is sufficient to see cutoff, see also Lemma \ref{lem:ProductCriterion} for a formal statement of their result. A first example of a family of trees with cutoff was found by Peres and Sousi \cite{PS:CutoffTree}. Note that although their results are stated for a discrete-time model of the simple random walk, they can be easily converted to the continuous-time setup using \cite{CS:DiscreteVSContinuous}. We will re-obtain the result of Peres and Sousi using Theorem \ref{thm:CutoffMiclo}, see Corollary \ref{cor:CutoffTree}, and Corollary \ref{cor:NewTreesWithCutoff} for simplified example of a family of trees with cutoff. \\

\vspace{-0.5mm}In the following, we will focus on random walks on finite trees and assume that the underlying graphs will always be a sequence of rooted trees. 
Before we come to the main theorems, we introduce some notations. We let $\ell(v)$ denote the set of edges in the shortest path of a vertex $v \in V$ from the root $o$ and write $|v|:=|\ell(v)|$. For every edge $e\in E$, we let $e_{-},e_{+} \in V$ with $|e_{-}|<|e_{+}|$ denote the endpoints of the edge and set $|e|:=|e_{+}|$.
Further, let $T_v$ be the largest subtree of $G$ rooted at $v$ which consists only of sites with distance at least $|v|$ from $o$, and we use the conventions that $T_e:=T_{e_{+}}$ and $|T_e|:= |V(T_{e_+})|$ for all $e\in E$. We say that a vertex $x$ is a $\boldsymbol{\delta}$\textbf{-center of mass} if there are two trees $T$ and $\tilde{T}$ of $G$ with $V(T) \cap V(\tilde{T}) = \{x\}$ and $V(T) \cup V(\tilde{T}) = V(G)$ such that 
\begin{equation}\label{def:CenterOfMass}
|V(T)|\geq \delta|V(G)| ,\qquad |V(\tilde{T})|\geq\delta |V(G)| \ .
\end{equation}  The existence of a $\delta$-center of mass is guaranteed for all $\delta \in [0,\frac{1}{3}]$ by the following result. For its proof, we refer to the appendix.
\begin{proposition}\label{pro:CenterOfMass}
For all trees $G$, there must be at least one vertex which is a $\frac{1}{3}$-center of mass. 
\end{proposition}
From now on, for a given family of trees $(G_n)$, we fix some $\varepsilon>0$ and assume without loss of generality that the root $o_n$ is chosen to be a $\delta$-center of mass for some $\delta>0$ which does not depend on $n$.  For a pair of real-valued, non-negative sequences $(f(n))_{n \geq 0}$ and $(g(n))_{n \geq 0}$, we write $f \asymp g$ if $f$ and $g$ have the same order of magnitude, and we write  $f \ll g$ if $f(n)/g(n)$ converges to $0$ for $n$ going to infinity. \\

\vspace{-0.5mm}We start with a criterion on trees which allows us to determine when cutoff does not occur for the  simple random walk. 
\begin{theorem}\label{thm:NoCutoff} For a sequence of trees $(G_n)$ suppose that
\begin{equation}\label{eq:NoCutoffCriterion}
\max_{e\in E(G_n)} |e| |T_e| \asymp \max_{v \in V(G_n)} \sum_{e \in \ell(v)} |T_e|\, .
\end{equation} Then the simple random walk on $(G_n)$ does not exhibit cutoff.
\end{theorem}
Using Theorem \ref{thm:NoCutoff}, one can directly check that for example the simple random walk on the segment of size $n$ or on the regular tree truncated at level $n$ does not exhibit cutoff. In order to state a converse theorem, which guarantees the existence of cutoff, we require the following definition. Fix a tree $G$ and a vertex $v \neq o$. 
Let $v_0=o,v_1,\dots, v_k=v$ denote the sites in $\ell(v)$ for $k=|v|$. For each $v_i$, let $\bar{T}_i$ be  the largest subtree which is attached to $v_i$ and does not intersect with $\ell(v)$ in any site other than $v_i$. A $\boldsymbol{v}$\textbf{-retraction} of $G$ is a tree $\tilde{G}$ rooted at some site $\tilde{o}$, which consists of a segment of size $|\ell(v)|$ with $\tilde{o}$ at one of its endpoints. In addition, we let a binary tree of size $|\bar{T}_i|$ emerge from each site in the segment at distance $i$ (if $|\bar{T}_i|$ is not a power of two, consider the binary tree with leaves missing at the last level). 
%A visualization of this construction is given in Figure \ref{fig:RetractionCap}.
\begin{figure} \label{fig:Retraction}
\begin{center}
\begin{tikzpicture}[scale=0.7]

	\node[shape=circle,scale=1,fill] (A) at (0,0){} ;
 	\node[shape=circle,scale=1,fill] (B) at (3,0){} ;
	\node[shape=circle,scale=1,fill] (C) at (5.5,0){} ;
 	\node[shape=circle,scale=1,fill] (D) at (7.5,0){} ;	
 	
 	  	\draw[line width=1.7pt,RoyalBlue,dashed] (A) to (B);		
 		\draw[line width=1.7pt,RoyalBlue,dashed] (B) to (C);		
 		\draw[line width=1.7pt,RoyalBlue,dashed] (C) to (D);

	\node[shape=circle,scale=1,fill] (AX) at (11,0){} ;
 	\node[shape=circle,scale=1,fill] (BX) at (14,0){} ;
	\node[shape=circle,scale=1,fill] (CX) at (16.5,0){} ;
 	\node[shape=circle,scale=1,fill] (DX) at (18.5,0){} ;		

 	  	\draw[line width=1.7pt,RoyalBlue,dashed] (AX) to (BX);		
 		\draw[line width=1.7pt,RoyalBlue,dashed] (BX) to (CX);		
 		\draw[line width=1.7pt,RoyalBlue,dashed] (CX) to (DX);

 	\node[shape=circle,scale=1,fill] (A1) at (0,-1.5){} ;
	\node[shape=circle,scale=1,fill] (A2) at (0,-3){} ;
 	\node[shape=circle,scale=1,fill] (A3) at (-0.8,-4.5){} ;		
	\node[shape=circle,scale=1,fill] (A4) at (0,-4.5){} ;
 	\node[shape=circle,scale=1,fill] (A5) at (0.8,-4.5){} ;

	\node[shape=circle,scale=1,fill] (B1) at (3-0.4,-1.5){} ;
 	\node[shape=circle,scale=1,fill] (B2) at (3.4,-1.5){} ;
	\node[shape=circle,scale=1,fill] (B3) at (3-0.4,-3){} ;
 	\node[shape=circle,scale=1,fill] (B4) at (3.4,-3){} ;

	\node[shape=circle,scale=1,fill] (C1) at (5.5,-1.5){} ;
 	\node[shape=circle,scale=1,fill] (C2) at (5.5,-3){} ;

	\node[shape=circle,scale=1,fill] (D1) at (7.5,-1.5){} ;

 	  	\draw[thick] (A) to (A1);		
 		\draw[thick] (A1) to (A2);		
 		\draw[thick] (A2) to (A3);	
 		\draw[thick] (A2) to (A4);	
 		\draw[thick] (A2) to (A5);	
 	  	\draw[thick] (B) to (B1);		
 		\draw[thick] (B) to (B2);		
 		\draw[thick] (B1) to (B3);	
 		\draw[thick] (B2) to (B4);	
  		\draw[thick] (C) to (C1);		
 		\draw[thick] (C1) to (C2);	 	
   		\draw[thick] (D) to (D1);

	\draw (0,-0.55) to [closed, curve through = {(-0.25,-0.7)(-0.7,-5)(0,-5.1)(0.7,-5)(0.25,-0.7)}] (0,-0.55);
 	 \draw (3,-0.55) to [closed, curve through = {(3-0.25,-0.7)(3-0.7,-3.5)(3,-3.6)(3.7,-3.5)(3.25,-0.7)}] (3,-0.55);
 	 \draw (5.5,-0.55) to [closed, curve through = {(5.5-0.25,-0.7)(5.5-0.3,-3.5)(5.5,-3.55)(5.8,-3.5)(5.75,-0.7)}] (5.5,-0.55); 	
 
  	 \draw (7.5,-0.55) to [closed, curve through = {(7.5-0.25,-0.7)(7.5-0.3,-2.2)(7.5,-2.25)(7.8,-2.2)(7.75,-0.7)}] (7.5,-0.55);

 	 	\node[scale=1]  at (0.95,-0.5){$v_0=o$};
 	 	\node[scale=1]  at (3.5,-0.5){$v_1$};
 	 	\node[scale=1]  at (6,-0.5){$v_2$};
 		\node[scale=1]  at (8.5,-0.5){$v_3=v$};

		\node[scale=0.9]  at (-1,-3.5){$\bar{T}_0$};
		\node[scale=0.9]  at (3,-4){$\bar{T}_1$};
		\node[scale=0.9]  at (5.5,-4){$\bar{T}_2$};
		\node[scale=0.9]  at (7.5,-2.7){$\bar{T}_3$};

 	\node[shape=circle,scale=1,fill] (AX1) at (11-0.4,-1.5){} ;
	\node[shape=circle,scale=1,fill] (AX2) at (11.4,-1.5){} ;
 	\node[shape=circle,scale=1,fill] (AX3) at (11-0.8,-3){} ;		
	\node[shape=circle,scale=1,fill] (AX4) at (11,-3){} ;
 	\node[shape=circle,scale=1,fill] (AX5) at (11.8,-3){} ;

	\node[shape=circle,scale=1,fill] (BX1) at (14-0.4,-1.5){} ;
 	\node[shape=circle,scale=1,fill] (BX2) at (14.4,-1.5){} ;
	\node[shape=circle,scale=1,fill] (BX3) at (14-0.8,-3){} ;
 	\node[shape=circle,scale=1,fill] (BX4) at (14,-3){} ;

	\node[shape=circle,scale=1,fill] (CX1) at (16.1,-1.5){} ;
 	\node[shape=circle,scale=1,fill] (CX2) at (16.9,-1.5){} ;

	\node[shape=circle,scale=1,fill] (DX1) at (18.5,-1.5){} ;

		\draw[thick] (AX) to (AX1);		
 		\draw[thick] (AX) to (AX2);		
 		\draw[thick] (AX1) to (AX3);	
 		\draw[thick] (AX1) to (AX4);	
 		\draw[thick] (AX2) to (AX5);	
 	  	\draw[thick] (BX) to (BX1);		
 		\draw[thick] (BX) to (BX2);		
 		\draw[thick] (BX1) to (BX3);	
 		\draw[thick] (BX1) to (BX4);	
  		\draw[thick] (CX) to (CX1);		
 		\draw[thick] (CX) to (CX2);	 	
   		\draw[thick] (DX) to (DX1);

	\draw (0,-0.55) to [closed, curve through = {(-0.25,-0.7)(-0.7,-5)(0,-5.1)(0.7,-5)(0.25,-0.7)}] (0,-0.55);
 	 \draw (3,-0.55) to [closed, curve through = {(3-0.25,-0.7)(3-0.7,-3.5)(3,-3.6)(3.7,-3.5)(3.25,-0.7)}] (3,-0.55);
 	 \draw (5.5,-0.55) to [closed, curve through = {(5.5-0.25,-0.7)(5.5-0.3,-3.5)(5.5,-3.55)(5.8,-3.5)(5.75,-0.7)}] (5.5,-0.55); 	
 
  	 \draw (7.5,-0.55) to [closed, curve through = {(7.5-0.25,-0.7)(7.5-0.3,-2.2)(7.5,-2.25)(7.8,-2.2)(7.75,-0.7)}] (7.5,-0.55);

		 \node[scale=1]  at (11.5,-0.5){$\tilde{o}$};
 		\node[scale=1]  at (19,-0.5){$v$};

\end{tikzpicture}
\end{center}
\caption{\label{fig:RetractionCap} Example of a tree $G$ on the left-hand side and its corresponding $v$-retraction $\tilde{G}$ on the right-hand side. The edges in the segment $\ell(v)$ connecting the root and the site $v$ are depicted in dashed blue.}

\end{figure}
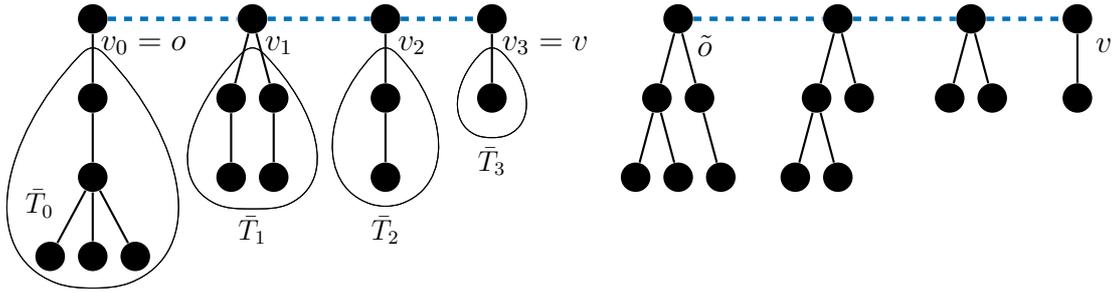

\begin{theorem}\label{thm:CutoffCompactification} For every $G_n$ in $(G_n)_{n \geq 1}$, suppose that we find some $v^{\ast}_n \in V(G_n)$ with
\begin{equation}\label{eq:MaximalPath}
\sum_{e \in \ell(v^{\ast}_n)} |T_e| \asymp \max_{v \in V(G_n)} \sum_{e \in \ell(v)} |T_e| \gg |V(G_n)|  \, .
\end{equation} If we have that 
\begin{equation}\label{eq:HyperbolicGrowth}
\max_{e\in \ell(v^{\ast}_n)} |e| |T_e| \ll \max_{v \in V(G_n)} \sum_{e \in \ell(v)} |T_e|
\end{equation} then for any sequence $(\tilde{G}_n)$ of $v^{\ast}_n$-retractions of $(G_n)$, we have that the simple random walk on $(\tilde{G}_n)$ exhibits cutoff.
\end{theorem}
We will see that taking retractions is necessary for cutoff in Theorem \ref{thm:CutoffCompactification} when we study random walks on spherically symmetric trees in Section \ref{sec:Spherically}, see Remark \ref{rem:NecessaryRetraction}. Note that instead of attaching binary trees, we may allow also for attaching other classes of graphs with sufficiently fast growth, see Remark \ref{rem:AlphaRetraction} for precise conditions on the attached trees.
We will now give a simple example of a family of trees on which the simple random walk exhibits  cutoff. 
This follows as an immediate consequence of Theorem \ref{thm:CutoffCompactification}.
 \begin{corollary}\label{cor:NewTreesWithCutoff} Consider the family of trees $(G_n)_{n\geq 1}$, where $G_n$ consists of a segment of size $n$, rooted at one of its endpoints, and binary trees of size $\lfloor n/(i+1)^{2}\rfloor$ attached at distance $i$ from the root for all $i\in \{0,1,\dots,n\}$. Then the root is a $\delta$-center of mass for $\delta=\frac{1}{6}$ and the simple random walk on $(G_n)$ exhibits cutoff. 
 \end{corollary}
Our last main result is a criterion which is particularly suited to verify cutoff for \textit{thin} and \textit{long} trees, see also the trees constructed in \cite{PS:CutoffTree}. For all $k \geq 0$, we set
\begin{equation}\label{V_m}
V_k=V_k(G):= \bigcup_{v \colon |v|= k} V(T_v) \ .
\end{equation}

\begin{theorem}\label{thm:CutoffMiclo} Recall \eqref{V_m} and suppose that a family of trees $(G_n)$ with maximum degrees $(\Delta_n)$ satisfies
\begin{equation}\label{eq:ComplimentAssumption}
\max_{k \geq 1} k | V_k(G_n) | \ll \frac{1}{\Delta_n}\max_{v \in V(G_n)} \sum_{e \in \ell(v)} |T_e| \ . 
\end{equation} Then the simple random walk on $(G_n)$ exhibits cutoff.
\end{theorem}

\subsection{Outline of the paper} \label{sec:Outline}

This paper will be organized as follows. In Section \ref{sec:PrelimSpectral}, we discuss preliminary estimates on the spectral gap of the simple random walk on trees. We present two bounds on the spectral gap, a first bound using a characterization of the spectral gap via a discrete Hardy inequality, and a second bound using weighted paths which follows directly from the first characterization. In Section \ref{sec:PrelimMixing}, we present preliminary facts on upper and lower bounds on the mixing time using a representation as hitting times of large sets. Building on these results, we prove our main criteria for the occurrence of cutoff in Section \ref{sec:Proofs}. Section \ref{sec:CutoffAtypical} is dedicated to applying these criteria to the families of trees in Theorem \ref{thm:CutoffExamples}, and showing that the simple random walk does not exhibit cutoff. In particular, we verify the absence of cutoff for the simple random walk on spherically symmetric trees, supercritical Galton--Watson trees conditioned on survival, and families of trees which converge to the Brownian continuum random tree. The latter includes Galton--Watson trees conditioned to contain a certain number of sites. We conclude with an outlook on open problems.

\section{Some facts about the spectral gap on trees} \label{sec:PrelimSpectral}

In order to study cutoff for the simple random walk on trees, our key tool will be to give bounds on the spectral gap of the random walk $(\eta_t)_{t \geq 0}$. Let $\mathcal{L}$ denote the generator of the random walk on a tree $G=(V,E)$. We study pairs of eigenvalues and eigenfunctions for the random walk, i.e.\ we want to find $(\mu,f)$ for $\mu \in \mathbb{C}$ and $f \colon V \rightarrow \mathbb{C}$ such that
\begin{equation}
(\mathcal{L}f)(x) = \mu f(x)
\end{equation} holds for all $x \in V$. Note that since $(\eta_t)_{t \geq 0}$ is reversible, all eigenvalues of $\mathcal{L}$ are real-valued. Moreover, the function $f \equiv 1$ is always an eigenfunction with respect to the eigenvalue $\lambda=0$. Our goal is to investigate the \textbf{spectral gap} $\lambda$ of the process, i.e.\ the absolute value of the second largest eigenvalue of $\mathcal{L}$, respectively the \textbf{relaxation time} 
\begin{equation}\label{relax}
t_{\textup{rel}}:=\frac{1}{\lambda}.
\end{equation}
Recall the following variational characterization of $\lambda$, see Definition 2.1.3 in \cite{S:SaintFlour}.

\begin{lemma}\label{lem:Rayleigh} Let $\lambda$ be the spectral gap of simple random walk $(\eta_t)_{t \geq 0}$ on the tree $G$. Then we have that
\begin{equation}
\lambda = \min_{f \colon V \rightarrow \R, \Var(f)\neq 0} \frac{\mathcal{E}(f)}{\Var(f)}
\end{equation} where we set
\begin{equation*}\label{def:DirichletVariance}
\mathcal{E}(f):=\frac{1}{|V|}\sum_{e\in E} \left( f(e_{+})-f(e_{-}) \right)^2,  \quad \Var\left( f\right) := \frac{1}{|V|}\sum_{v \in V} f(v)^2 - \frac{1}{|V|^2}\Big( \sum_{v \in V} f(v)\Big)^2  .
\end{equation*}
\end{lemma} The quantities $\mathcal{E}(f)$ and $\Var\left( f\right)$ are the Dirichlet form and the variance of the function $f$, see Chapter 13 of \cite{LPW:markov-mixing} for a more general introduction. 

\subsection{A discrete Hardy inequality on trees} \label{sec:Hardy}

Using Lemma \ref{lem:Rayleigh}, we obtain the following characterization of the spectral gap in terms of a (discrete) Hardy inequality on trees, see also \cite{EHP:Hardy}.
\begin{lemma}\label{lem:EquivalentHardy} Recall that $o$ is a $\delta$-center of mass for some $\delta>0$, and let $T$ and $\tilde{T}$ be two trees which satisfy \eqref{def:CenterOfMass}. Let $A$ be the smallest constant such that we have
\begin{equation}\label{eq:HardyTree1}
\sum_{v\in V(T)} \Big(\sum_{e \in \ell(v)} g(e)\Big)^2 \leq A \sum_{e \in E(T)} g(e)^2
\end{equation} for all functions $g \colon E(T) \rightarrow \R$ with $g \not\equiv 0$ as well as
\begin{equation}\label{eq:HardyTree2}
\sum_{v\in V(\tilde T)} \Big(\sum_{e \in \ell(v)} \tilde{g}(e)\Big)^2 \leq A \sum_{e \in E(\tilde T)} \tilde{g}(e)^2
\end{equation} for all functions $\tilde{g} \colon E(\tilde{T}) \rightarrow \R$ with $\tilde{g} \not\equiv 0$. Then we have that the spectral gap $\lambda$ of the simple random walk on $G$ satisfies
\begin{equation}\label{eq:RelationAtoSpectralGap}
\lambda \in \left[ \frac{1}{A},\frac{1}{\delta A}\right]  \ .
\end{equation}
\end{lemma} 
\begin{proof} For any function $f\colon V(G) \rightarrow \R$, we set
\begin{equation}
f_T(v):= (f(v)-f(o)) 1_{v \in V(T)}, \qquad  f_{\tilde{T}}(v):= (f(v)-f(o)) 1_{v \in V(\tilde{T})}
\end{equation} for all $v \in V(G)$. Using the definition of $\Var(f)$, we see that
\begin{equation}\label{eq:SplitDirichlet}
\frac{\mathcal{E}(f)}{\Var(f)} \geq \frac{\mathcal{E}(f_T)+ \mathcal{E}( f_{\tilde{T}})}{\bar{f}_T+ \bar{f}_{\tilde{T}}} \geq \min\left\{ \frac{\mathcal{E}(f_T)}{\bar{f}_T} , \frac{\mathcal{E}( f_{\tilde{T}})}{ \bar{f}_{\tilde{T}}} \right\}
\end{equation} where we set
\begin{equation}
 \bar{f}_{T} := \frac{1}{|V(G)|} \sum_{v\in V(T) } (f_T(v))^2 ,\qquad  \bar{f}_{\tilde T} := \frac{1}{|V(G)|} \sum_{v\in V(\tilde T) } (f_{\tilde{T}}(v))^2 \ . 
\end{equation} We set $g_T(e):=f_T(e_{+})-f_T(e_{-})$ for $e\in E(T)$ and $g_{\tilde{T}}(e):=f_{\tilde{T}}(e_{+})-f_{\tilde{T}}(e_{-})$ for $e\in E({\tilde{T}})$. The inequalities \eqref{eq:HardyTree1} and  \eqref{eq:HardyTree2} applied to $g_T$ and $g_{\tilde{T}}$ yield 
$\mathcal{E}(f_T)\geq \frac{1}{A}\bar{f}_T $ and $\mathcal{E}( f_{\tilde{T}})  \geq \frac{1}{A}\bar{f}_{\tilde{T}}$. Now, we use Lemma \ref{lem:Rayleigh} and \eqref{eq:SplitDirichlet} to conclude the lower bound on $\lambda$ which is claimed in \eqref{eq:RelationAtoSpectralGap}. \\

For the corresponding upper bound on $\lambda$, we note that the minimizer $g^{\ast}$ for $A$ is given by \eqref{eq:HardyTree1} with equality and satisfies $g^{\ast}(e) \geq 0$ for all $e \in E(T)$ (consider $|g^{\ast}|$ otherwise to see that this has to be the case). Define a function $h:V(G) \rightarrow \R$ by
\begin{equation}
h(v) := 1_{v\in V(T)}\sum_{e\in \ell(v)} g^{\ast}(e) \ .
\end{equation} Recall that $\pi$ is the uniform measure on $V(G)$ and use the Paley–Zygmund inequality to see that
\begin{equation*}
\Var(h) \geq (1- \pi(h>0))\frac{1}{|V(G)|} \sum_{v \in V(G)} h(v)^2 
\end{equation*}
and hence, we must have 
\begin{equation*}
\Var(h) \geq    \pi(h=0) \mathcal{E}(h)A  \geq \delta A\mathcal{E}(h) \ .
\end{equation*} Using the characterization of $\lambda$ in Lemma \ref{lem:Rayleigh}, this concludes the proof.
\end{proof}

 We will use Lemma \ref{lem:EquivalentHardy} in Sections \ref{sec:BoundRelaxation} and \ref{sec:CriterionMiclo} in order to obtain upper and lower bounds on the spectral gap, respectively.

 \subsection{A bound on the relaxation time by weighted paths} \label{sec:WeightedPaths}
 
 Next, we give an estimate which allows us to achieve upper bounds on the relaxation time for the simple random walk on trees. Note that this bound was obtained already in \cite{LMW:SpectralGapTrees} in a more general setup, including a corresponding lower bound of the same form within a factor of $2$. For the convenience of the reader, we provide a short proof for our special case of the simple random walk.
Recall \eqref{relax}.
\begin{proposition} \label{pro:UpperBoundRelaxationTime} Let $(a_e)_{e \in E(G)}$ be any family of positive edge weights for the tree $G$. For any choice of the $(a_e)$'s, we have that
\begin{equation}\label{eq:UpperBoundRelax}
t_{\textup{rel}} \leq \max_{e \in E(G)}a_e^{-1}\sum_{v \in T_e} \sum_{\tilde{e} \in \ell(v)} a_{\tilde{e}} \ .
\end{equation} 
\end{proposition} 
\begin{proof} Recall the trees $T,\tilde{T}$ for $G$ from \eqref{def:CenterOfMass}. By Lemma \ref{lem:Rayleigh} and Lemma \ref{lem:EquivalentHardy}, if we have
for some constant $C$ that
\begin{equation}\label{eq:forsomeC}
\sum_{v\in V(T)} \Big(\sum_{e \in \ell(v)} g(e)\Big)^2 \leq C  \sum_{e \in E(T)} g(e)^2
\end{equation} for all functions $g \colon E(T) \rightarrow \R$, and a similar statement with respect to $\tilde{T}$, we conclude $t_{\textup{rel}} \leq C$. In the following, we only show \eqref{eq:forsomeC} with respect to the tree $T$. Fix positive edge weights $a_e$ and note that by the Cauchy-Schwarz inequality,
\begin{equation*}
\sum_{v\in V(T)} \Big(\sum_{e \in \ell(v)} g(e)\Big)^2 \leq  \sum_{v\in V(T)}  \Big( \sum_{\tilde{e} \in \ell(v)} a_{\tilde{e}}\Big) \sum_{e \in \ell(v)} a_e^{-1}g(e)^2 \ .
\end{equation*} Rearranging the sum according to the edges and assuming without loss of generality that $g$ is non-negative, we get 
\begin{align*}
\sum_{v\in V(T)} \Big(\sum_{e \in \ell(v)} g(e)\Big)^2 \leq \sum_{e \in E(T)} \Big( \sum_{v \in T_{e}}   \sum_{\tilde{e} \in \ell(v) } a_{\tilde{e}} \Big) a^{-1}_{e}  g(e)^2 \, .
\end{align*}
Hence, for any choice of $(a_e)$, 
%using Lemma \ref{lem:EquivalentHardy} and the definition of the constant $A$,
we can take the right-hand side of \eqref{eq:UpperBoundRelax} as $C$ in \eqref{eq:forsomeC}.
\end{proof}  Note that we can choose positive weights $(a_e)$ in any particular way to obtain an upper bound. In the following, we present three special cases of $(a_e)$'s and the respective bounds on the relaxation time. We start with the case, where we set $a_e=\frac{1}{|e|}$ in Proposition \ref{pro:UpperBoundRelaxationTime} to obtain the following bound.
\begin{corollary} \label{cor:SimpleBound} We have that
\begin{equation}
t_{\textup{rel}} \leq (\log(\textup{diam}(G))+1) \max_{e\in E} |e| |T_e|
\end{equation} where $\textup{diam}(G)$ denotes the diameter of the tree $G$.
\end{corollary}
In Proposition \ref{pro:LowerBoundRelaxationTime}, we give a corresponding lower bound on the relaxation time, but without the additional factor of $\log(\textup{diam}(G))$. Next, consider $a_e=\frac{1}{f(|e|)}$ for some function $f \colon \N \rightarrow \R$, whose reciprocals are summable. We obtain the following bound as a consequence of Proposition \ref{pro:UpperBoundRelaxationTime}.
 \begin{corollary} \label{cor:SimpleBound2} Suppose that we have some function $f$ on the integers with 
 \begin{equation*}
 \sum_{n \in \N} f(n)^{-1} = C < \infty
 \end{equation*} Then we have that
\begin{equation}
t_{\textup{rel}} \leq C \max_{e\in E} f(|e|) |T_e| \ .
\end{equation} 
\end{corollary}  Next, we have the following bound by choosing the weights $a_e=|T_e|$ for all $e\in E(T)$ proportional to size of the tree $T_e$ attached tree to $e$.
\begin{corollary}  \label{cor:SimpleBound3} It holds that
\begin{equation}\label{eq:SimpleBound3}
t_{\textup{rel}} \leq  \max_{v \in V} \sum_{e \in \ell(v)} |T_e| \ .
\end{equation} 
\end{corollary}
This follows directly from Proposition \ref{pro:UpperBoundRelaxationTime} noting that
\begin{equation*}
t_{\textup{rel}} \leq \max_{e \in E} \frac{1}{|T_e|} \sum_{v \in T_{e}} \sum_{\tilde{e} \in \ell(v)} |T_{\tilde{e}}| \leq  \max_{e \in E} \frac{1}{|T_e|} \sum_{v \in T_{e}} \max_{\tilde{v} \in V} \sum_{\tilde{e} \in \ell(\tilde{v})} |T_{\tilde{e}}| =  \max_{v \in V} \sum_{e \in \ell(v)} |T_e|   \ .
\end{equation*}
We will see in the upcoming section that the right-hand side in \eqref{eq:SimpleBound3} gives also an upper bound on the mixing time for the simple random walk.

\section{Some facts about the mixing time on trees}  \label{sec:PrelimMixing}

In this section, we discuss mixing time estimates for the simple random walk $(\eta_t)_{t\geq 0}$ on a tree $G=(V,E,o)$. The main result presented in this section is that the $\varepsilon$-mixing time can be bounded in terms of the hitting time of the root $o$. For sites $x,y \in V$, let
\begin{equation}
\tau_{\textup{hit}}(x) := \inf \left\{ t \geq 0 \colon \eta_t = x \right\}
\end{equation} be the \textbf{hitting time} of $x$ and let $\E_{y}[\tau_{\textup{hit}}(x)]$ denote the expected hitting time of $x$ when starting the random walk from $y$. The following proposition gives an upper bound on the $\varepsilon$-mixing time.
 \begin{proposition} \label{pro:HittingtimeUpper} 
 Let $t_{\textup{mix}}(\varepsilon)$ be the $\varepsilon$-mixing time of $(\eta_t)_{t\geq 0}$. There is a universal constant $C>0$  such that we have for all $\varepsilon \in (0,\frac{1}{2})$
 \begin{equation}\label{eq:hittingTimeEstimateUpper}
t_{\textup{mix}}(\varepsilon) \leq  C \log(\varepsilon^{-1})\Big(1+ \max_{v \in V} \E_{v}[\tau_{\textup{hit}}(o)]\Big)  \leq   2C \log(\varepsilon^{-1})\max_{v \in V} \sum_{e \in \ell(v)} |T_e| .
 \end{equation} 
In particular,
\begin{equation}\label{mixdiam}
t_{\textup{mix}}(\varepsilon)  \leq 2C \log(\varepsilon^{-1})|V|\textup{diam}(G).
\end{equation}
 \end{proposition}
 \begin{proof} Note that the first inequality in  \eqref{eq:hittingTimeEstimateUpper} is immediate from Theorem 5 in  \cite{A:MixingHitting} together with Theorem 20.3 in \cite{LPW:markov-mixing} and Lemma 5.2 of \cite{PS:MixingHitting}. To see that the second inequality holds, we claim that for two adjacent sites $v,w \in V$ with $|v|<|w|$, we have 
 \begin{equation}\label{eq:HittingDegree}
 \E_{w}[\tau_{\textup{hit}}(v)] \leq  \sum_{x \in T_w} \deg(x) -1 \ . 
 \end{equation} This follows from the fact that for the embedded discrete-time simple random walk 
\begin{equation}\label{eq:DiscreteSRW}
 \E_{w}[\tau_{\textup{hit}}(v)] =  \sum_{x \in T_w} \deg(x) -1  \ ,
 \end{equation} 
and a time-change argument, see Section 2.3 in \cite{B:Center}. Since for any tree $\tilde{G}$, we have
 \begin{equation}\label{eq:SitesToEdges}
 \sum_{x \in V(\tilde{G})} \deg(x) = 2|V(\tilde{G})|-2 \ ,
 \end{equation} this yields
\begin{equation*}
\E_{v}[\tau_{\textup{hit}}(o)] \leq 2 \sum_{e \in \ell(v)} |T_e| \ ,
\end{equation*}
and the second inequality in \eqref{eq:hittingTimeEstimateUpper} follows.
 \end{proof}
 Note that the bound in Proposition \ref{pro:HittingtimeUpper} does not use the assumption that the root is a $\delta$-center of mass. In fact, we see that \eqref{eq:hittingTimeEstimateUpper} holds for an arbitrarily chosen root. Next, we state a lower bound on the mixing time, following the ideas of Lemma 5.4 in \cite{BHP:CutoffTrees}, which does indeed require that the root is a $\delta$-center of mass. 
 \begin{proposition} \label{pro:HittingtimeLower} Assume that the root $o$ is a $\delta$-center of mass for some $\delta>0$. Let $\Delta$ be the maximum degree in $G$. Then for all $\varepsilon\leq\delta$,
 \begin{equation}\label{eq:hittingTimeEstimateLower}
t_{\textup{mix}}\left(\frac{\varepsilon}{2}\right) \geq \frac{\varepsilon}{2} \max_{v \in V}\E_v\left[\tau_{\textup{hit}}(o) \right]  \geq  \frac{\varepsilon}{2 \Delta} \max_{v \in V} \sum_{e \in \ell(v)} |T_e| \ .
 \end{equation}
 \end{proposition}
 \begin{proof}
Let $v \in V$ with $v \neq o$ be fixed and recall that $\pi$ denotes the uniform distribution on $V(G)$. Moreover, recall the trees $T$ and $\tilde{T}$ in \eqref{def:CenterOfMass} from the definition of the $\delta$-center of mass and assume that $v \in V(T)$. We now claim that
\begin{equation}\label{eq:GeometricBound}
\P_v\left(\tau_o \leq t_{\textup{mix}}\left({\varepsilon}/{2}\right) \right) \geq \P_v\Big( \eta_{t_{\textup{mix}}({\varepsilon}/{2})} \in V(\tilde{T})\Big) \geq \pi(V(\tilde{T})) - \frac{\varepsilon}{2} \geq \frac{\varepsilon}{2}\, .
\end{equation} 
The first inequality in \eqref{eq:GeometricBound} follows since $\eta_{t_{\textup{mix}}({\varepsilon}/{2})} \in V(\tilde{T})$ implies that the root was visited before $t_{\textup{mix}}\left({\varepsilon}/{2}\right)$, the second inequality uses the definition of the mixing time and the third inequality follows since $\pi(V(\tilde{T}))\geq  \delta \geq \varepsilon$.
In words,  \eqref{eq:GeometricBound} says that with probability at least $\varepsilon/2$, the random walk hits the root until time $t_{\textup{mix}}(\varepsilon/2)$. Using the Markov property of $(\eta_t)_{t \geq 0}$, since \eqref{eq:GeometricBound} holds for any $v\in V(T)$,
we can iterate \eqref{eq:GeometricBound} to get that the probability of hitting the root for the first time until time 
$k \cdot t_{\textup{mix}}(\varepsilon/2)$ is at least the probability that a
Geometric-$(\varepsilon/2)$-random variable takes a value $\leq k$. This yields
\begin{equation*}
\E_v\left[\tau_{\textup{hit}}(o) \right] 
%\leq \frac{1}{\P_v\left(\tau_{\textup{hit}}(o) \leq t_{\textup{mix}}\left(\varepsilon/{2}%\right)\right)}   t_{\textup{mix}}\left(\frac{\varepsilon}{2}\right) 
\leq \frac{2}{\varepsilon} t_{\textup{mix}}\left(\frac{\varepsilon}{2}\right)\, .
\end{equation*} Since $v\in T$ was arbitrary, we obtain the first inequality in \eqref{eq:hittingTimeEstimateLower}. For the second inequality, recall \eqref{eq:DiscreteSRW} and \eqref{eq:SitesToEdges}, and use a comparison with a discrete-time simple random walk which is sped up by a factor of $\Delta$.
 \end{proof} So far, we gathered techniques in order to give upper and lower bounds on the mixing and relaxation time. The next seminal result by Basu et al.\ gives a characterization of cutoff for the simple random walks on trees, see \cite{BHP:CutoffTrees}.
 
\begin{lemma}[c.f.\ Theorem 1 in \cite{BHP:CutoffTrees}] \label{lem:ProductCriterion} Fix $\varepsilon \in (0,1)$. Let $(G_n)_{n \geq 1}$ be a family of trees with $\varepsilon$-mixing times $(t^n_{\textup{mix}})$ and relaxation times $(t^n_{\textup{rel}})$ for the simple random walk on $(G_n)$, respectively. The simple random walk on $(G_n)$ exhibits cutoff if and only if
\begin{equation}\label{seminalcomp}
t^n_{\textup{rel}} \ll t^n_{\textup{mix}} \ .
\end{equation} In other words, we see cutoff whenever then product criterion  \eqref{eq:ProductCriterionIntro} holds.
\end{lemma}

\begin{remark} \label{rem:StrategyCutoff}
In view of  Corollary \ref{cor:SimpleBound3}, Proposition \ref{pro:HittingtimeLower} and Lemma \ref{lem:ProductCriterion}, we aim at verifying cutoff by giving weights $(a_e)$ in Proposition \ref{pro:UpperBoundRelaxationTime}, which lead to a strictly improved bound over the weights $a_e = |T_e|$.
\end{remark}

\section{Proof of the main criteria for cutoff} \label{sec:Proofs}

We now use the preliminary bounds from the previous two sections in order to establish our main criteria on cutoff for the simple random walk on trees, i.e.\ we prove Theorem \ref{thm:NoCutoff}, Theorem \ref{thm:CutoffCompactification}, and Theorem \ref{thm:CutoffMiclo}.

\subsection{A general lower bound on the relaxation time} \label{sec:BoundRelaxation}

We start with a lower bound on the relaxation time. Recall that for a family of graphs $(G_n)$, the root is chosen to be a $\delta$-center of mass for some fixed $\delta>0$. 
 
\begin{proposition} \label{pro:LowerBoundRelaxationTime} The constant $A$ from Lemma \ref{lem:EquivalentHardy} satisfies
\begin{equation}\label{eq:LowerBoundRelax}
A \geq \max_{e \in E(G)} |e| |T_e| \ .
\end{equation} In particular, Lemma \ref{lem:EquivalentHardy} implies 
that 
\begin{equation}
t_{\textup{rel}}\geq \delta \max_{e \in E(G)} |e| |T_e|\, .
\end{equation}
\end{proposition} 
\begin{proof} We show the lower bound on $A$ by considering an explicit function $g_{e^{\ast}}$ in the definition of $A$ in \eqref{eq:HardyTree1} and \eqref{eq:HardyTree2}. Fix some $e^{\ast}\in E(G)$ which maximizes the right-hand side of \eqref{eq:LowerBoundRelax}. We define $g_{e^{\ast}}: E(G) \rightarrow \R$ to be 
\begin{equation*}
g_{e^{\ast}}(e) := \begin{cases}
\frac{1}{|e^{\ast}_+|} &  \text{ if } e \in \ell(e^{\ast}_+) \\
0 & \text{ if } e \notin \ell(e^{\ast}_+)
\end{cases}
\end{equation*} and note that $g_{e^{\ast}}$ satisfies
\begin{equation*}
\sum_{v\in V(G)} \Big(\sum_{e \in \ell(v)} g_{e^{\ast}}(e)\Big)^2 \geq \sum_{v\in T_{e^{\ast}}} \Big(\sum_{e \in \ell(v)} g_{e^{\ast}}(e)\Big)^2 = |T_{e^{\ast}}|
\end{equation*} as well as
\begin{equation*}
\sum_{e \in E(T)} g_{e^{\ast}}(e)^2 = \frac{1}{|e^{\ast}_+|} = \frac{1}{|e^{\ast}|} \ .
\end{equation*}  Using the definition of $A$ in Lemma \ref{lem:EquivalentHardy}, we 
see that \eqref{eq:LowerBoundRelax} holds.
\end{proof} 
\begin{proof}[Proof of Theorem \ref{thm:NoCutoff}] Recall that by Lemma \ref{lem:ProductCriterion}, there is no cutoff when the mixing time and the relaxation time are of the same order. Note that  
Proposition \ref{pro:HittingtimeUpper} yields an upper bound on the mixing time, while Proposition \ref{pro:LowerBoundRelaxationTime} establishes a lower bound on the relaxation time. Both bounds are of the same order, due to assumption \eqref{eq:NoCutoffCriterion}, which finishes the proof.
\end{proof}

\subsection{Cutoff for \textit{v}-retractions of trees} \label{sec:CutoffRetraction}

In this section, we prove Theorem \ref{thm:CutoffCompactification}. We use Proposition \ref{pro:UpperBoundRelaxationTime} with a specific choice of edge weights $(a_e)$ improving the choice of weights leading to Corollary \ref{cor:SimpleBound3}.

\begin{lemma}\label{lem:CutoffWeights} Suppose that assumptions \eqref{eq:MaximalPath} and \eqref{eq:HyperbolicGrowth} of Theorem \ref{thm:CutoffCompactification} hold and recall that $(\tilde{G}_n)_{n \geq 1}$ denotes the sequence of $v_n^{\ast}$-retractions of $(G_n)_{n \geq 1}$. Then we have 
\begin{equation}\label{eq:LemmaCutoffWeights}
\tilde{t}^n_{\textup{rel}} \ll \max_{v \in V(\tilde{G}_n)} \sum_{e \in \ell(v)} |T_e|  
\end{equation} where $\tilde{t}^n_{\textup{rel}}$ is the relaxation time belonging to $\tilde{G}_n$.
\end{lemma}
\begin{proof} Let $k := |v_n^{\ast}|$ and let $(\bar{T}_i)_{i \in \{0,\dots,k\}}$ denote the trees in $\tilde{G}_n$ attached along the segment $\ell(v_n^{\ast})$, ordered according to their distances from the root, which are used in the construction of the  $v_n^{\ast}$-retraction. We consider the following choice of the weights $(a_e)_{e \in E(G_n)}$. For $e \in \ell(v_n^{\ast})$, we let $a_e:= |e|^{-1/2}$. For $e \in E(\bar{T}_i)$, we set 
\begin{equation}
 a_e:= \frac{ 1}{\sqrt{\max\{i,1\} } (|e|-i)^2} \ .
\end{equation}
In order to apply Proposition \ref{pro:UpperBoundRelaxationTime}, we will now give an upper bound for
\begin{equation}\label{eq:VariationBoundCutoffWeights}
A_e := a_e^{-1}\sum_{v \in T_e} \sum_{\tilde{e} \in \ell(v)} a_{\tilde{e}}
\end{equation}
 for any possible choice of the edge $e \in E(\tilde{G}_n)$. We claim that without loss of generality, it suffices to consider $e\in \ell(v_n^{\ast})$. To see this, consider $e\in \bar{T}_{i}$ for $i\geq 1$ and let $e_i$ be the corresponding edge in $\ell(v_n^{\ast})$ with $|e_i|=i$. For every $v\in \bar{T}_i$, we have 
 \begin{equation}\label{eq:SitesInSubtrees}
 \sum_{\tilde{e} \in \ell(v)} a_{\tilde{e}} = \sum_{j=1}^{i} j^{-\frac{1}{2}} + i^{-\frac{1}{2}} \sum_{l=i+1}^{|e|} \frac{1}{(l-i)^2} \in [\sqrt{i}, 4 \sqrt{i}] \ .
 \end{equation}
Together with the fact that $\bar{T}_{i}$ has a binary structure, we get that
 \begin{align*}
A_{e}
 \leq 4 i(|e|-i)^2  |T_e| % \\
\leq 4 i(|e|-i)^2 2^{-(|e|-i)} |T_{e_i}| \leq 5 i  |T_{e_i}|    \leq 5  A_{e_i}
 \end{align*} holds. Similarly, for $e \in \bar{T}_0$, we see that $A_e \leq 5 |\bar{T}_0| \leq 5|V(G_n)|$, which is of smaller order than the right-hand side in \eqref{eq:LemmaCutoffWeights} by assumption \eqref{eq:MaximalPath}.
Hence, it suffices to bound $A_e$ for all edges $e$ within $\ell(v_n^{\ast})$ to conclude. From \eqref{eq:SitesInSubtrees}, we see that the right-hand side of \eqref{eq:VariationBoundCutoffWeights} is bounded from above by
\begin{equation*}
A_e = \sqrt{|e|} \sum_{i=|e|}^{k}\sum_{v \in \bar{T}_i} \sum_{\tilde{e} \in \ell(v)} a_{\tilde{e}} \leq 4\sqrt{|e|} \sum_{i=|e|}^{k} \sqrt{i} |\bar{T}_i| \ .
\end{equation*} 
It remains to show that
\begin{equation}\label{eq:InterpolationEstimate}
\sqrt{|e|} \sum_{i=|e|}^{k} \sqrt{i} |\bar{T}_i|  \ll  \sum_{e \in \ell(v_n^{\ast})} |\bar{T}_e|  \, .
\end{equation} Using the Cauchy-Schwarz inequality and \eqref{eq:HyperbolicGrowth}, we see that for all $e \in \ell(v_n^{\ast})$,
\begin{equation}\label{eq:CSInequality}
\Big(\sqrt{|e|} \sum_{i=|e|}^{k} \sqrt{i} |\bar{T}_i| \Big)^2 \leq \Big(|e|\sum_{i=|e|}^{k} |\bar{T}_i| \Big)\Big(\sum_{i=|e|}^{k} i |\bar{T}_i| \Big)  \ll \Big(\sum_{i=1}^{k} i |\bar{T}_i| \Big)^2 \, .
\end{equation} Taking square roots of both sides of \eqref{eq:CSInequality} finishes the proof.
\end{proof}
\begin{proof}[Proof of Theorem \ref{thm:CutoffCompactification}]
For the lower bound on the mixing time from Proposition \ref{pro:HittingtimeLower}, note that the leading order does not decrease when we replace $(G_n)$ by $(\tilde{G}_n)$ for $v_n^{\ast}$ from \eqref{eq:MaximalPath}, and that the root is still a $\tilde{\delta}$-center of mass for some $\tilde{\delta}>0$. Theorem \ref{thm:CutoffCompactification} follows together with the product criterion from Lemma \ref{lem:ProductCriterion}, and Lemma \ref{lem:CutoffWeights}. 
\end{proof}
\begin{remark}\label{rem:AlphaRetraction} Note that instead of binary trees, we may consider other trees $\bar{T}_i$ rooted at $v_i$  and attached to the segment at distance $i$ from $\tilde{o}$ for which
\begin{equation}\label{eq:RetractionAssumptionRem}
\max_{v \in V(T_{e})} \sum_{\bar{e} \in \ell(v)} |T_{\bar{e}}| \leq \alpha |T_{e}| 
\end{equation} holds for some constant $\alpha$. The above choice of binary trees satisfies \eqref{eq:RetractionAssumptionRem} with $\alpha = 2$. Depending on %the orders of growth of the quantities in 
\eqref{eq:HyperbolicGrowth}, we may also allow $\alpha$ to go to infinity with $n$ sufficiently slow.
\end{remark}

\subsection{A sufficient condition for cutoff on trees} \label{sec:CriterionMiclo}

In this section, we give an upper bound on the relaxation time, which will allow us to give a sufficient condition for cutoff on trees. For a finite rooted tree $G=(V,E,o)$, we let $\mathcal{S}:=\{ B \subseteq V \colon B \neq \emptyset, o \notin B\}$, and define for all $B \in \mathcal{S}$
\begin{equation}\label{def:alphaW}
\nu_{B}:= \inf \Big\{ \sum_{e\in E} f(e)^2 \colon \sum_{e \in \ell(v)} f(e) \geq 1 \text{ for all } v \in B \Big\} > 0 \ ,
\end{equation} where the infimum is taken over functions $f\colon E \rightarrow \R$. For the following bound on the relaxation time, we use the ideas of Evans et al.\ for proving a Hardy inequality on continuous weighted trees \cite{EHP:Hardy}. A corresponding Hardy inequality for discrete weighted trees was obtained by Miclo, see Proposition 16 in \cite{M:SpectrumMarkovChains}. We will now provide a similar  result on the relaxation time of the simple random walk. 
\begin{proposition}\label{pro:UpperBoundRelaxationTimeHardy} 
Recall \eqref{V_m}. For any finite tree $G$, we have that 
\begin{equation}\label{eq:EstimateHardy}
 t_{\textup{rel}} \leq 32 \max_{B \in \mathcal{S}} \frac{|B|}{\nu_B} \leq 32   \max_{k \in \N} k |V_{k}| \, .
\end{equation} 
\end{proposition}
\begin{proof} Recall the trees $T,\tilde{T}$ for $G$ from \eqref{def:CenterOfMass}. For the first inequality in \eqref{eq:EstimateHardy}, it suffices by Lemma \ref{lem:EquivalentHardy} to show that for every function $g \colon E(T) \rightarrow \R$ with $g \not\equiv0$, we have
\begin{equation}\label{eq:ForSomeMax}
\sum_{v\in V(T)} \Big(\sum_{e \in \ell(v)} g(e)\Big)^2 \leq 32\left(\max_{B \in \mathcal{S}} \frac{|B|}{\nu_B} \right) \sum_{e \in E(T)} g(e)^2 \ ,
\end{equation} and a similar statement with respect to $\tilde{T}$. We will now show \eqref{eq:ForSomeMax} only for $T$.  Fix a non-negative function $g$ (otherwise consider $|g|$), and define for all $i\in \Z$ the set
\begin{equation*}
B_i := \Big\{ v \in V(T) \colon 2^{i} \leq \sum_{e \in \ell(v)}g(e) < 2^{i+1} \Big\} \ .
\end{equation*}
We let $\mathcal{I}:= \{ i \in \Z \colon B_i \neq \emptyset \}$ and note that $B_i \in \mathcal{S}$ when $i \in  \mathcal{I}$. Observe that
\begin{equation*}
\sum_{v\in V(T)} \Big(\sum_{e \in \ell(v)} g(e)\Big)^2 \leq \sum_{i \in \mathcal{I}} 2^{2i+2} |B_i| \leq \left(\max_{B \in \mathcal{S}} \frac{|B|}{\nu_B}  \right) \sum_{i \in \mathcal{I}} 2^{2i+2} \nu_{B_i} 
\end{equation*} holds. Hence, it suffices now to prove that for all $i \in \mathcal{I}$
\begin{equation}\label{eq:BiEstimate}
\nu_{B_i} \leq 2^{-2i+2} \sum_{e\in E(T) \colon e_+ \in B_{i-1}\cup B_i}g(e)^2 
\end{equation}
is satisfied in order to conclude \eqref{eq:ForSomeMax}. To see that \eqref{eq:BiEstimate} indeed holds, consider 
\begin{equation*}
g_i(e) := 2^{-i+1}g(e)\mathds{1}_{\{e_+ \in B_{i-1} \cup B_i\}}
\end{equation*} for all $e\in E(T)$. Using the definition of the sets $(B_i)_{i \in \Z}$ and the fact that $g$ is non-negative, we have for all $v\in B_i$ \begin{equation*}
\sum_{e \in \ell(v)} g_i(e) = 2^{-i+1}\Big(\sum_{e \in \ell(v)} g(e) - \sum_{\tilde{e} \in \ell(v) \colon \tilde{e}_+ \notin B_{i-1}\cup B_i} g(\tilde{e})\Big)\geq 2^{-i+1}\big(2^{i}- 2^{i-1}\big) = 1\ . 
\end{equation*} Plugging $g_i$ into the definition \eqref{def:alphaW} of $\nu_{B_i}$ yields \eqref{eq:BiEstimate}, and hence the first inequality in \eqref{eq:EstimateHardy}.  
For the second inequality in \eqref{eq:EstimateHardy}, observe that for all $B\in \mathcal{S}$ 
\begin{equation}\label{eq:BChoice}
\frac{\nu_{B}}{|B|} \geq \frac{1}{|B|}\max_{v \in B} \inf\Big\{ \sum_{e \in E(G)} f(e)^2 \colon \sum_{e \in \ell(v)} f(e)=1 \Big\} = \frac{1}{|B|}\max_{v \in B} \frac{1}{|v|} 
\end{equation} holds, where the infimum is attained when $f(e)=|v|^{-1} \mathds{1}_{\{e\in \ell(v)\}}$. We then optimize in \eqref{eq:BChoice} over the choice of $B\in \mathcal{S}$  to conclude.
\end{proof}
\begin{proof}[Proof of Theorem \ref{thm:CutoffMiclo}] Similar to the proof of Theorem \ref{thm:NoCutoff}, recall that by Lemma \ref{lem:ProductCriterion}, we have cutoff if and only if the relaxation time is of strictly smaller order than the mixing time. Note that by Proposition \ref{pro:HittingtimeLower} we obtain a lower bound on the mixing time, while Proposition \ref{pro:UpperBoundRelaxationTimeHardy} gives us an upper bound on the relaxation time. We conclude taking into account assumption \eqref{eq:ComplimentAssumption}.
\end{proof}
As a direct consequence of Theorem \ref{thm:CutoffMiclo}, we see that the simple random walk on the following family of trees constructed by Peres and Sousi in \cite{PS:CutoffTree} exhibits cutoff: 
Fix $k$ and consider the tree $G_N$ for $N=N(k)$ which is constructed as follows: Set $n_i = 2^{2^{i}}$ and start with a segment of length $n_k$, rooted at one of the endpoints. At the root, we attach a binary tree of size $N:=n_k^3$. Then at distance $n_i$ from the root, attach a binary tree of size $N/{n_i}$ for all $i \in \{ k/2,\dots,k\}$. We reindex the sequence of trees $(G_N)$ to obtain a family of trees $(G_n)_{n \geq 1}$ for all $n\geq 1$. A visualization of these trees can be found in \cite{PS:CutoffTree}. Note that in this construction, the root is always a $\delta$-center of mass for some $\delta>0$ which does not depend on $n$. 
 \begin{corollary}\label{cor:CutoffTree} The family of trees described by Peres and Sousi has a mixing time of order $Nk$ and a relaxation time of order $N$, and hence exhibits cutoff provided that $k \rightarrow \infty$ with $N \rightarrow \infty$.
 \end{corollary}
 \begin{proof} We first consider the $\varepsilon$-mixing time for some $\varepsilon>0$. In Proposition \ref{pro:HittingtimeLower}, we choose $v$ to be a leaf in the tree attached at distance $n_k$ from the root and get
 \begin{equation*}
 t_{\textup{mix}}(\varepsilon) \geq c\sum_{i = k/2}^{k-1} (n_{i+1}-n_{i})\sum_{j=i}^{k-1} \frac{N}{n_j} \geq \frac{c}{4} k N 
 \end{equation*} for some constant $c=c(\varepsilon)>0$, when summing only over the sites in the attached binary trees for \eqref{eq:hittingTimeEstimateLower}. Since $N/n_i \gg n_i$ for all $i \leq k$, a corresponding upper bound follows from Proposition \ref{pro:HittingtimeUpper}. While a lower bound on $t_{\textup{rel}}$ of order $N$ is immediate from Proposition \ref{pro:LowerBoundRelaxationTime}, we obtain a corresponding upper bound of the same order by Proposition \ref{pro:UpperBoundRelaxationTimeHardy}, by a straight-forward count of the sites of distance at least $m$ from the root. We conclude that cutoff occurs by Lemma \ref{lem:ProductCriterion}.
% Note that since
%\begin{equation}
% \sum_{j \geq i}^{k} \frac{N}{n_j} +n_j \leq 2 \frac{N}{n_i} \ ,  \quad \log\left( N\right) \leq n_i
%\end{equation} holds for all $i \in \{ k/2,\dots,k\}$ and $k$ sufficiently large, we obtain that
%\begin{equation}
%\Big| \bigcup_{|x|=n} T_x \Big| \leq 3 \max_{i \colon 2n_i \geq m} \frac{N}{n_i} \leq 6 \frac{N}{m}
%\end{equation} which finishes the proof. 
 \end{proof}
%Again, we refer to Corollary \ref{cor:NewTreesWithCutoff} for a simplified example of a family of trees on which the simple random walk exhibits cutoff.

\section{Cutoff for SRW on trees is atypical} \label{sec:CutoffAtypical}

We now use the previous estimates and ask about cutoff for the simple random walk on the following three families of trees from Theorem \ref{thm:CutoffExamples}: 
infinite spherically symmetric trees truncated at level $n$, supercritical Galton--Watson trees conditioned on survival and truncated at height $n$, and families of combinatorial trees converging to the Brownian continuum random tree (CRT). The latter includes critical Galton--Watson trees conditioned to have exactly $n$ sites. In all three cases, we will in the following verify that cutoff  does not occur.

\subsection{Spherically symmetric trees} \label{sec:Spherically}

Let $G=(V,E,o)$ be a rooted tree. We say that $G$ is \textbf{spherically symmetric} if we have that $\deg(v)=\deg(v^{\prime})$ holds for all $v,v^{\prime} \in V$ with $|v|=|v^{\prime}|$. Examples for such trees are the integer lattice or regular trees. We write $\deg_k$ for the degree of the vertices at generation $k$ provided that $V_{k}\neq \emptyset$ holds (recall \eqref{V_m}). 

\begin{proposition}\label{pro:NoCutoffSpherically} For a given infinite spherically symmetric tree $G$ of maximum degree $\Delta$, let $(G_n)_{n \geq 1}$ be the family of trees, induced by $G$ restricted to the sites $V \setminus V_n$ for all $n\in \N$. Then the simple random walk on $(G_n)_{n \geq 1}$ does not exhibit cutoff.
\end{proposition}

When $G=\N$ the claim is well-known, see \cite{LPW:markov-mixing}. Otherwise, for all $n\in \N$ large enough, we choose the root of $G_n$ to be the first branching point in $G$, i.e.\ the vertex closest to the root with degree at least $3$. In particular, note that the root of $G_n$ will be a $\frac{1}{4}$-center of mass for all $n$ sufficiently large. By Proposition \ref{pro:HittingtimeUpper} and \ref{pro:HittingtimeLower}, we see that the $\frac{1}{8}$-mixing time of $G_n$ satisfies
\begin{equation}\label{eq:HittingTimeSpherically}
t^n_{\textup{mix}}\left(\frac{1}{8}\right) \asymp \sum_{i=1}^{n-1} \sum_{j=i+1}^{n-1} \prod_{k=i}^{j-1} (\deg_k -1)  \ .
\end{equation}
We will now show a lower bound on the relaxation time of the same order by using a comparison to birth-and-death chains. Our strategy will be similar to \cite{NN:SpectrumTree}, where Nestoridi and Nguyen study eigenvalues and eigenfunctions for the $d$--regular tree. More precisely, we will exploit that certain eigenfunctions of birth-and-death chains can be converted to eigenfunctions of the simple random walk on spherically symmetric trees. Without loss of generality, we will assume that $\deg(o)>1$ holds since removing the sites before the first branching point in $G$ changes the relaxation time by at most a constant factor. This can be seen  using the characterization of the spectral gap in Lemma \ref{lem:Rayleigh}.
Let $(X_t)_{t \geq 0}$ be the continuous-time birth-and-death chain on the segment $\{1,\dots,2n-1\}$ with nearest neighbor transition rates $(r(x,y))$ given by
\begin{equation}
r(x,y)=\begin{cases} \deg_{x-n}-1 & \text{ if } y=x+1>n+1 \\
\deg_{n-x}-1 & \text{ if } y=x-1<n-1 \\
1 & \text{ otherwise}
\end{cases}
\end{equation} for all $\{x,y\} \in E(G_n)$. We make the following observation on the spectral gap of $(X_t)_{t \geq 0}$.
\begin{lemma}\label{lem:EigenfunctionBirthDeath} Let $\tilde{\lambda}$ be the spectral gap of $(X_t)_{t \geq 0}$. There exists a corresponding eigenfunction $\tilde{f} \colon \{1,\dots,2n-1\} \rightarrow \R$ which satisfies \begin{equation}\label{eq:EigenfunctionsBirthDeath}
\tilde{f}(x)=-\tilde{f}(2n-x)
\end{equation} for all $x \in \{1,\dots,2n-1\}$.
\end{lemma}
\begin{proof} Note that there exists an eigenfunction $g$ corresponding to the spectral gap of an irreducible birth-and-death chain which is monotone, see Lemma 22.17 in \cite{LPW:markov-mixing}. By the symmetry of the transition rates, we see that the function $h$ given by $h(x)=g(2n-x)$ for all $x \in \{1,\dots,2n-1\}$ is also an eigenfunction for $(X_t)_{t \geq 0}$ with respect to $\tilde{\lambda}$. Since $g(1) \neq g(2n-1)$, we see that the function $\tilde{f}:=g-h$ is an eigenfunction corresponding to $\tilde{\lambda}$ which has the desired properties.
\end{proof}
With these observations, we have all tools to give the proof of Proposition \ref{pro:NoCutoffSpherically}.
\begin{proof}[Proof of Proposition \ref{pro:NoCutoffSpherically}] 
Observe that we can extend the eigenfunction $\tilde{f}$ of $(X_t)_{t \geq 0}$ from Lemma \ref{lem:EigenfunctionBirthDeath}  belonging to the spectral gap to an eigenfunction $F$ with eigenvalue $\tilde{\lambda}$ of the random walk on $G_n$. To see this, let $x_1,x_2$ be two sites adjacent to $o$, and consider the function $F \colon V \rightarrow \R$  given by
\begin{equation}
F(v)=\begin{cases} f(n-|v|) & \text{ if } v \in T_{x_1}  \\
f(n+|v|) & \text{ if } v \in T_{x_2}  \\
0 & \text{ otherwise. }  
\end{cases}
\end{equation}
The fact that $F$ is an eigenfunction of the simple random walk on $G_n$ follows by a direct verification using the generator. 
Hence, in order to give a lower bound on the relaxation time of the same order as in \eqref{eq:HittingTimeSpherically}, it suffices to bound the spectral gap $\tilde{\lambda}$ of $(X_t)_{t \geq 0}$. Note that the stationary distribution $\tilde{\pi}$ of $(X_t)_{t \geq 0}$ satisfies
\begin{equation}
\tilde{\pi}(x) =\frac{1}{Z} \sum^{n-2}_{k=\min\{ x,2n-x\}} (\deg_{n-k-1} - 1)
\end{equation} for all $x\in \{ 1,\dots,2n-1\}$, where  $Z$ is a normalization constant. Recalling that $G$ has maximum degree $\Delta$ and using Theorem 4.2 in \cite{CS:SpectrumBirthDeath}, we see that
\begin{equation*}
\frac{1}{\tilde{\lambda}} \geq \frac{1}{16 \Delta} \sum_{i=1}^{n} \frac{1}{\tilde{\pi}(i)} \asymp  \sum_{i=1}^{n} \Big( \sum_{j=1}^n \prod_{k=j}^{n-2} (\deg_{n-k-1} - 1) \Big) \prod_{k=i}^{n-2} (\deg_{n-k-1} - 1)^{-1} \ .
\end{equation*} 
This gives the desired lower bound on the relaxation time of the tree $G_n$ which is of the same order as the upper bound on the mixing time in \eqref{eq:HittingTimeSpherically}. Hence, using Lemma \ref{lem:ProductCriterion}, we see that no cutoff occurs.
\end{proof}

\begin{remark}\label{rem:NecessaryRetraction} Using Proposition \ref{pro:NoCutoffSpherically}, we can see that taking $v_n^*$-retractions in Theorem \ref{thm:CutoffCompactification} is necessary for cutoff. More precisely, consider the spherically symmetric tree $G$ with 
\begin{equation}\deg_i = \begin{cases} 3 & \text{ if } i=2^{j}-1 \text{ for some } j \in \{0,1,\dots \} \\
2 & \text{ else}
\end{cases}
\end{equation} for all $i \in \N$. The corresponding trees $(G_n)_{n \geq 0}$ truncated at level $n$ satisfy \eqref{eq:HyperbolicGrowth}, but due to Proposition \ref{pro:NoCutoffSpherically}, the simple random walk on $(G_n)$ does not exhibit cutoff.
\end{remark}

\subsection{Galton--Watson trees conditioned on survival}  \label{sec:GaltonWatsonSurvival}

In this section we consider a family of random trees $(G_n)_{n \geq 1}$ which we obtain by truncating a supercritical Galton--Watson tree conditioned on survival. More precisely, let 
$\mu$ be an \textbf{offspring distribution} and assume that
\begin{equation}\label{eq:AssumptionsGWT1}
m:= \sum_{j=1}^{\infty}j \mu(j) >1 \qquad \sigma^2 :=\sum_{j=1}^{\infty}j^2 \mu(j)  \in(0,\infty), 
\end{equation} i.e.\ we have a supercritical Galton--Watson process whose offspring distribution has finite variance. In the following, we take the genealogical tree $G$ of a realization of the Galton--Watson process conditioned on survival. We then obtain the trees $(G_n)_{n \geq 1}$ by restricting $G$ onto the sites $V \setminus V_n$ (recalling \eqref{V_m}). We keep the sequence of trees fixed and  perform simple random walks on $(G_n)_{n \geq 1}$. We denote the law of $G$ by $P$ ($P$ depends on $\mu$, but we will not write it). We now have the following result on cutoff for the simple random walk on $(G_n)_{n \geq 1}$. 
\begin{proposition}\label{pro:NoCutoffGWT1} We have that for $P$-almost all trees $G$, the family of simple random walks on $(G_n)_{n \geq 1}$ does not exhibit cutoff.
\end{proposition}
\begin{proof}
From the proof of Theorem 1.4(b) in \cite{J:MixingInterchange}, we know that the relaxation time is $P$-almost surely of order $m^n$ for all $n$ sufficiently large. Hence, by Lemma \ref{lem:ProductCriterion}, it suffices to give an upper bound on the mixing time of order $m^n$ for the random walk on $G_n$ in order to exclude cutoff. Providing an upper bound on the mixing time of order $m^n$ answers a question of Jonasson \cite{J:MixingInterchange}. Recall that for any supercritical Galton--Watson tree $G=(V,E)$ with offspring distribution $\mu$ satisfying \eqref{eq:AssumptionsGWT1},
\begin{equation}\label{eq:MartingaleGenerations}
\left(  \frac{|\{v \in V(G) \colon |v|=n\}|}{m^{n}}\right)_{n \geq 1}
\end{equation} is a martingale which converges $P$-almost surely and in $L^2$, due to the Kesten-Stigum Theorem.
For all $v\in V$ and $N\geq 0$ let $k(v,N)$ be the number of sites in the tree  $T_v$ with distance at most $|v|+N$ from the root. From \eqref{eq:MartingaleGenerations}, we see that for all $v\in V$ 
\begin{equation}\label{def:MartingaleLimit}
\sup_{N \in  \N} \frac{k(v,N)}{m^{N}} = Y_v 
\end{equation} $P$-almost surely for some random variable $Y_v$. It is easy to show, applying Doob's inequality to the martingale in \eqref{eq:MartingaleGenerations}, that $Y_v$ has a finite second moment. \\

Recall that Proposition \ref{pro:HittingtimeUpper} does not require the root to be a $\delta$-center of mass, and hence 
\begin{equation}\label{eq:MixingBoundGWT1}
P\left( t^n_{\textup{mix}}\left(\frac{1}{4}\right) \leq T \right)\leq P\left( C\sum_{s =0}^{n-1} m^{n-s} \max_{|v|=s}Y_v \leq T \right)
\end{equation} holds $P$-almost surely for all $T>0$ and some constant $C>0$. It remains to give a bound on the random variables $Y_v$. Note that the random variables $Y_v$ do not depend on the number of sites in the generation $|v|$ in the Galton--Watson tree. Hence, writing $E$ for the expectation corresponding to $P$, we see that
\begin{align*}
 P\Big(\max_{|v|=s}Y_v > \frac{m^s}{2s^2} \Big) &\leq P\left(|\{v\colon |v|=s\}| > s^2m^s\right) + s^2m^n P\Big(Y_v > \frac{m^s}{s^2} \Big) \\
 &\leq \frac{E\left[ | \{v \colon |v|= s \}|\right]}{s^2m^{s}} + s^2m^{s}\frac{E\left[ Y^{2}_o\right]}{s^{-4}m^{2s}}\leq c s^{-2}
\end{align*} is satisfied for all $s\in \N$ and some constant $c>0$. Together with  \eqref{eq:MixingBoundGWT1} and the first Borel--Cantelli lemma, we see that $P$-almost surely
\begin{equation*}%\label{eq:MixingBoundGWT2}
 t^n_{\textup{mix}}\left(\frac{1}{4}\right) \leq \tilde{c} m^{n}
\end{equation*}  holds for some $\tilde{c}>0$ depending on $G$ and all $n$ sufficiently large. Note that a bound of the same form follows when conditioning the underlying Galton--Watson tree on survival, using its representation as a multi-type Galton--Watson process with one child in every generation having a size-biased offspring distribution, see Chapter 12 of \cite{LP:ProbOnTrees}. Together with the corresponding bound on the relaxation time for the random walk on $(G_n)$ of order $m^n$ from \cite{J:MixingInterchange}, this finishes the proof.
%Using \ref{eq:MartingaleGenerations}, we see that for a sequence of supercritical Galton--Watson trees conditioned on survival, we have
%\begin{equation*}
%\inf_{n \in  \N} \frac{|V(G_n)|}{\mathfrak{m}^{n}} = \tilde{W}_o
%\end{equation*} for some  $\P$-almost surely strictly positive $ \tilde{W}_o$, using the Kesten-Stigum Theorem. Together with \eqref{eq:MixingBoundGWT2} and Theorem \ref{thm:NoCutoff}, we conclude the proof.
\end{proof}

\begin{remark} For critical or subcritical Galton--Watson trees, i.e.\ when $m\leq 1$ holds in \eqref{eq:AssumptionsGWT1}, one can consider the family of random walks on $(G_n)$ which we obtain from the associated Kesten tree truncated at generation $n$, see \cite{K:KestenTree} for a formal definition of the tree and \cite{J:SurveyTrees} for a more comprehensive discussion. Note that the resulting tree consists of a segment of size $n$ with almost surely constant size trees attached. Using Theorem \ref{thm:NoCutoff}, one can show that there is $P$-almost surely no cutoff.
\end{remark}

\subsection{Combinatorial trees converging to the CRT} \label{sec:GaltonWatsonTotal}

In the previous two examples, we considered an infinite tree, which we truncated at generation $n$ in order to obtain the family of trees $(G_n)_{n \geq 1}$. We now take a different perspective and study the simple random walk on a sequence of random  trees, where we assume that $G_n$ has exactly $n$ sites. \\
%but in return allow for some more flexibility in the structure of the trees. \\

For each tree $(G_n)$, we assign a labeling to the $n$ sites chosen uniformly at random, and declare the vertex with label $1$ to be the root of $(G_n)$. Let $s \colon \{ 1, \dots, 2n-1\} \rightarrow V_n$ be the \textbf{contour function} of $G_n$, which is given as the walk on $V_n$ when performing depth-first search on $G_n$, giving priority to sites with a smaller label, see also Section 2.6 in \cite{A:CRT3}. Intuitively, we obtain $s$ by embedding $G_n$ into the plane such that the shortest path distance is preserved and sites with a common ancestor are ordered increasingly according to their labels, see Figure \ref{fig:ContourCap}. We will write $|s(\cdot)|$ for the distance of $s(\cdot)$ to the root, and call $|s|\colon \{ 1, \dots, 2n-1\} \rightarrow\N $ again contour function, with a slight abuse of notation.
%We then follow the contour of $G_n$ with a 
%walker moving at speed one. 
For $c>0$, consider the normalized contour function $\tilde{s}:[0,1] \rightarrow \R$ given by
\begin{equation}
\tilde{s}_n\left(\frac{i}{2n}\right):=c n^{-\frac{1}{2}}|s(i)|
\end{equation} for all $i\in \{1,\dots,2n-1\}$, $\tilde{s}_n(0)=\tilde{s}_n(1)=0$ and linear interpolation between these values. 
We say that a family of random trees $(G_n)_{n \geq 1}$ \textbf{converges to the CRT for \textit{c}} if $(\tilde{s}_n)_{n \geq 1}$ converges in distribution (in the space $\mathcal{C}([0,1])$) to the Brownian excursion $(e(t))_{t \in [0,1]}$.
Note that the terminology CRT refers to the Brownian continuum random tree, which can be seen as the limit object of $(G_n)_{n \geq 1}$, see \cite{A:CRT1,A:CRT2, A:CRT3} for an introduction and equivalent definitions, and \cite{HM:ScalingUnlabelled,MM:ScalingBinary} for examples of trees converging to the CRT. \\

Arguably the most prominent example of such a sequence of trees are independently chosen critical Galton--Watson trees 
$(G_n)_{n\geq 1}$ conditioned on having exactly $n$ sites, where we assume that
\begin{equation}\label{eq:AssumptionsGWT2}
\sum_{j=1}^{\infty}j \mu(j) =1 \qquad \sum_{j=1}^{\infty}j^2 \mu(j) =:\sigma^2 \in(0,\infty) \qquad \gcd(j>0 \colon \mu(j)>0)=1
\end{equation} holds, i.e.\ the number of offspring has mean $1$, finite variance and is not supported on a sublattice of the integers. Note that the assumption of having $m=1$ is not a restriction as we can transform any given offspring distribution with positive mean into a critical offspring distribution without changing the law of the resulting graphs when conditioning on a given number of sites, see Section 2.1 in \cite{A:CRT2}. The following lemma is a classical result due to Aldous.
\begin{lemma}[c.f.\ Theorem 23 in \cite{A:CRT3}]\label{lem:AldousScaling} Under the assumptions in \eqref{eq:AssumptionsGWT2}, we have that $(G_n)_{n\geq 1}$ converges to the CRT for  $c=\sigma/2$.
% in the normalized contour function.
\end{lemma}
Note that for critical Galton--Watson trees, mixing and relaxation times were studied by Jonasson in \cite{J:MixingInterchange}, relying on estimates of the effective resistance between the root and the leaves. We now present an alternative proof of his results, which extends to more general families of combinatorial trees. In the following, we first choose the trees $(G_n)_{n \geq 1}$, keep them fixed and then perform simple random walks on $(G_n)_{n \geq 1}$. We denote the law of $(G_n)_{n \geq 1}$ by $P$.
\begin{proposition}\label{pro:NoCutoffGWT2} Let $(G_n)_{n \geq 1}$ be a family of random trees converging to the Brownian CRT for some $c$. Then we have that $P$-almost surely, the family of simple random walks on $(G_n)_{n \geq 1}$ does not exhibit cutoff.
\end{proposition}
\begin{figure} \label{fig:Contour}
\begin{center}
\begin{tikzpicture}[scale=0.7,	
  declare function={
    func(\x)= (\x < 1) * (0)   +
              and(\x >= 1, \x < 4) * (\x-1)     +
              and(\x >= 4, \x < 5) * (6-\x+1)     +
              and(\x >= 5, \x < 6) * (\x-2-1)     +
              and(\x >= 6, \x < 7) * (8-\x+1)     +
              and(\x >= 7, \x < 9) * (\x-4-1)     +
              and(\x >= 9, \x < 13) * (12-\x+1)     +
              and(\x >= 13, \x < 15) * (\x-12-1)     +
              and(\x >= 15, \x < 17) * (16-\x+1)     +
              and(\x >= 17, \x < 20) * (\x-16-1)     +
              and(\x >= 20, \x < 23) * (22-\x+1)     +
              and(\x >= 23, \x < 25) * (\x-22-1)     +
              and(\x >= 25, \x < 27) * (26-\x+1)     +
              (\x >= 27) * (0)  ;
  }
]

 	\node[shape=circle,scale=0.8,draw]  (B) at (3,0){$1$} ;
 	
 	\node[shape=circle,scale=0.8,draw] (A1) at (0,1.5){$3$} ;
	\node[shape=circle,scale=0.68,draw] (A2) at (0,3){$13$} ;
 	\node[shape=circle,scale=0.8,draw] (A3) at (-1.1,4.5){$2$} ;		
	\node[shape=circle,scale=0.8,draw] (A4) at (0,4.5){$9$} ;
 	\node[shape=circle,scale=0.68,draw] (A5) at (1.1,4.5){$10$} ;		
 	
	\node[shape=circle,scale=0.8,draw] (B1) at (3-0.8,1.5){$4$} ;
 	\node[shape=circle,scale=0.8,draw] (B2) at (3.8,1.5){$7$} ;
	\node[shape=circle,scale=0.68,draw] (B3) at (3-0.8,3){$12$} ;
 	\node[shape=circle,scale=0.8,draw] (B4) at (3.8,3){$5$} ;		
 	  	
 	\node[shape=circle,scale=0.68,draw] (C1) at (1.1,6){$11$} ;
	\node[shape=circle,scale=0.8,draw] (C2) at (3.8,4.5){$8$} ;
 	\node[shape=circle,scale=0.68,draw] (C3) at (6,1.5){$14$} ;		
	\node[shape=circle,scale=0.8,draw] (C4) at (6,3){$6$} ;

 	  	\draw[thick] (A1) to (B);		
	
 		\draw[thick] (A1) to (A2);		
 		\draw[thick] (A2) to (A3);	
 		\draw[thick] (A2) to (A4);	
 		\draw[thick] (A2) to (A5);	
 		
 	  	\draw[thick] (B) to (B1);		
 		\draw[thick] (B) to (B2);		
 		\draw[thick] (B1) to (B3);	
 		\draw[thick] (B2) to (B4);		
 	 	
 		\draw[thick] (A5) to (C1);		
 		\draw[thick] (B4) to (C2);		
 		\draw[thick] (B) to (C3);		
 		\draw[thick] (C3) to (C4);

%\draw (3,0) to [closed, dotted, curve through = {(-0.3,1.5)(-1.4,5)(-1.4,5.5)(1.4,5.5)(0,6.3)}] (3,0);
% 	 \draw (3,0.55) to [closed, curve through = {(3-0.25,0.7)(3-0.7,3.5)(3,3.6)(3.7,3.5)(3.25,0.7)}] (3,0.55);
% 	 \draw (5.5,0.55) to [closed, curve through = {(5.5-0.25,0.7)(5.5-0.3,3.5)(5.5,3.55)(5.8,3.5)(5.75,0.7)}] (5.5,0.55); 	
 
  %	 \draw (7.5,0.55) to [closed, curve through = {(7.5-0.25,0.7)(7.5-0.3,2.2)(7.5,2.25)(7.8,2.2)(7.75,0.7)}] (7.5,0.55); 			

 	 	\node[scale=1]  at (3.5,-0.5){$o$};
 %		\node[scale=1]  at (8,-0.5){$v$};
 
 	\node[scale=0.9]  at (10.1,5.7){$|s(i)|$};
 	\node[scale=0.9]  at (17,-0.5){$i$};
    \begin{axis}[xshift=9.5cm,
        domain=0:28,
        xmin=0, xmax=28,
        ymin=0, ymax=4.1, %ylabel=$s(i)$,
        samples=400, %xlabel=$i$,
        axis y line=center,
        axis x line=middle,
        scale=1.1
    ]

\addplot [RoyalBlue,thick] {func(x)};
\end{axis}

\end{tikzpicture}
\end{center}
\caption{\label{fig:ContourCap} Visualization of the linear interpolation of the contour function $|s| \colon \{1,\dots,27\} \rightarrow \N$ for the graph $G$ given on the left-hand side with $n=14$ sites. } 
\end{figure}
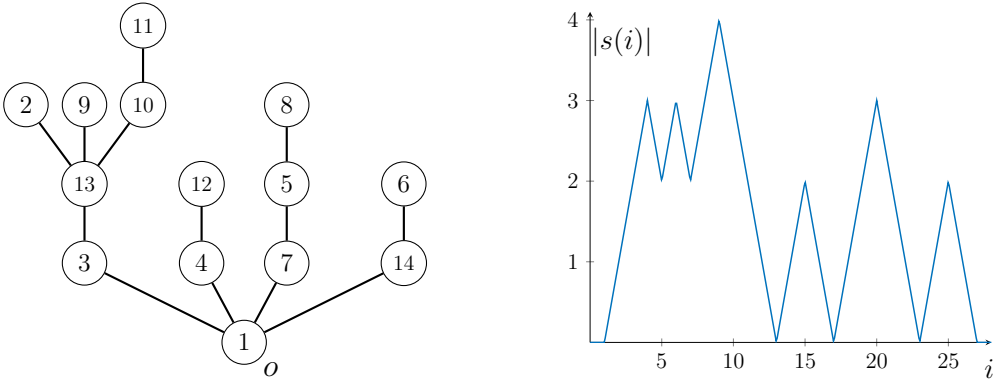
Intuitively, the absence of cutoff is explained by the fact that  the law of the random walk on $(G_n)_{n \geq 1}$ converges to the law of a Brownian motion on the CRT, indicating a smooth decay of the total-variation distance for the random walk at a time of order $N^{3/2}$.
In order to prove Proposition \ref{pro:NoCutoffGWT2}, we will show
that $P$-almost surely, \eqref{seminalcomp} is not satisfied. First, we will prove that the mixing time and the relaxation time are, for all $n$ sufficiently large, with positive probability both of order $n^{3/2}$. 
We start with the following upper bound on the mixing time.
\begin{lemma}\label{lem:MixingTimeGWT} Let $(G_n)$ be as in Proposition \ref{pro:NoCutoffGWT2}. Then,
for all $\varepsilon>0$, there exists a constant $C_{\varepsilon}$ such that for $n$ sufficiently large,
\begin{equation}
P\left( t^n_{\textup{mix}}\left(\frac{1}{4}\right) \leq  C_\varepsilon n^{\frac{3}{2}} \right) \geq 1 -  \varepsilon\, .
\end{equation} 
\end{lemma}
\begin{proof} Note that the tree $G_n$ satisfies $\textup{diam}(G_n) \leq 2 \max_{i \in \{1,\dots,2n-1\}}s(i)$.  From Lemma \ref{lem:AldousScaling}, we see that for all $\varepsilon>0$, there is $c_1 = c_1(\varepsilon) > 0$ such that for $n$ large enough,
\begin{equation}
P\Big(\textup{diam}(G_n) \geq  c_1 n^{\frac{1}{2}}\Big) < \varepsilon\, .
\end{equation} 
Now we use \eqref{mixdiam} from Proposition \ref{pro:HittingtimeUpper} to conclude.
%the fact that the $\frac{1}{4}$-mixing time can be bounded by 
%$c_2 |V_n|\textup{diam}(G_n)$ for some universal constant 
%$c_2>0$ by Proposition \ref{pro:HittingtimeUpper}, 
%see also Corollary 4.2 in \cite{NP:CriticalMixing} for general graphs.
\end{proof}
It remains to show a corresponding lower bound of order $n^{3/2}$ for the relaxation time of the simple random walk on $G_n$, which requires a bit of setup. We start by giving some statements for the Brownian excursion. We then carry over these observations to trees using Lemma \ref{lem:AldousScaling}. More precisely, we first choose a new root $o_{\ast}$, which will be a $\delta$-center of mass. We then show that there exists a site $v$ at distance of order $\sqrt{n}$ from the new root with order $n$ many sites in the tree $T_v$. Note that a Brownian excursion $(e(t))_{t \in [0,1]}$ attains almost surely its extrema in every compact subset of $[0,1]$. We let $t_{\min}$ and $t_{\max}$ be such that
\begin{equation}\label{def:ExtremValuesBE}
e(t_{\min}) = \min_{t \in \left[ \frac{1}{4},\frac{3}{4}\right]} e(t), \qquad e(t_{\max}) = \max_{t \in \left[ \frac{1}{4},\frac{3}{4}\right]} e(t) \ .
\end{equation} 
Consider, for $\theta> 0$ and $\delta > 0$, the events
\begin{equation}\label{eq:ExcursionArea}  B_1 := \left\lbrace e(t_{\min}) > \theta \right\rbrace \cap \left\lbrace \max_{t \in [0,\delta]} e(t) \leq \frac{\theta}{2}\right\rbrace
\end{equation}
and
\begin{equation}\label{eq:ExcursionBoundary}  
B_2 := 
\left\lbrace
 e(t) \geq 2\theta \text{ for all } t \in [t_{\max}-\theta ,t_{\max}+\theta] \right\rbrace \ ,
\end{equation}
see Figure \ref{fig:BrownianExCap} for a visualization.
\begin{lemma}\label{lem:ExcursionArea} For every $\varepsilon>0$, there exists some $\theta>0$ and some $\delta > 0$ 
such that $\P(B_1 \cap B_2 ) \geq 1- \varepsilon/2$.
\end{lemma}
\begin{proof} Note that $e(t_{\max}) > e(t_{\min})>0$ holds almost surely by the construction of the Brownian excursion. Since 
$(e(t))_{t \in [0,1]}$ is continuous, and $e(\frac{1}{4})$ and
$e(\frac{3}{4})$ have densities with respect to the Lebesgue measure on $[0, \infty)$, the probability of the event
$$
\left\lbrace e(t_{\min}) > \theta \right\rbrace \cap \left\lbrace
 e(t) \geq 2\theta \text{ for all } t \in [t_{\max}-\theta ,t_{\max}+\theta] \right\rbrace
$$ goes to $1$ for $\theta \to 0$. Hence, choose $\theta$ such that this probability is at least $1- \varepsilon/4$.
Moreover, since $e(0)=0$ almost surely and $(e(t))_{t \in [0,1]}$ has continuous paths, we have that
for every $\varepsilon,\theta>0$, there exists some $\delta>0$ such that the event  
$$
\left\lbrace \max_{t \in [0,\delta]} e(t) \leq \frac{\theta}{2}\right\rbrace
$$
has probability at least $1- \varepsilon/4$. This implies that the event $B_1\cap B_2$ has probability at least $1- \varepsilon/2$. 
\end{proof}

\begin{figure} \label{fig:BrownianEx} 
\begin{center}
\begin{tikzpicture}[scale=0.9]

\pgfmathsetseed{1332}

 \node (A1) at (0.01*rand,0.01*rand+1){};
 \node (A2) at (0.01*rand,0.01*rand+1){};
 \node (A3) at (0.01*rand,0.01*rand+1){};
 \node (A4) at (0.01*rand,0.01*rand+1){};
% \node (A5) at (rand,rand){};

\draw[RoyalBlue] (0,0) 
-- ++(0.02,0.15)
-- ++(0.02,-0.05)
-- ++(0.02,0.1)
-- ++(0.02,0.1)  
-- ++(0.02,-0.05)  
-- ++(0.02,+0.15)  
-- ++(0.02,-0.1)  
-- ++(0.02,0.2)  
-- ++(0.02,-0.15)  
-- ++(0.02,0.05)
\foreach \x in {1,...,740}
{   -- ++(0.02,-rand*0.2)
}-- ++(0.02,-0.23)     
-- ++(0.02,-0.15)  
-- ++(0.02,0)  
-- ++(0.02,-0.05)  
-- ++(0.02,0)  
-- ++(0.02,-0.05)
-- ++(0.02,-0.2)
-- ++(0.02,+0.05)
-- ++(0.02,-0.1);
\draw[thick] (0,0) -- (15.2,0);

\draw[thick,dashed] (15.2/4+0.32,0) -- (15.2/4+0.32,4);
\draw[thick,dashed] (3*15.2/4,0) -- (3*15.2/4,4);

\draw[thick] (15.2/4+0.32,0) -- (15.2/4+0.32,-0.3);
\draw[thick] (3*15.2/4,0) -- (3*15.2/4,-0.3);
\draw[thick] (15.2,0) -- (15.2,-0.3);
\draw[thick] (0,0) -- (0,-0.3);
%\draw[thick] (15.05,0) -- (15.05,-0.3);
\draw[thick] (0.15,0) -- (0.15,-0.3);

%\draw[thick,dotted] (15.05,0) -- (15.05,2);
\draw[thick,dotted] (0.15,0) -- (0.15,2);
%\draw[thick,dotted] (15.2,0) -- (15.2,2);
\draw[thick,dotted] (0,0) -- (0,2);

\draw (0,2) -- (0.5,2.5);
\draw (0.15,2) -- (0.5,2.5);

\node[scale=0.85] (Actilde) at (0.62,2.62){$\delta$} ;

\node[scale=3.42, rotate=270] (Abrac) at (7.17,2.7){$\}$} ;
\node[scale=0.85] (Amin35) at (7.17,2.3){$2\theta$} ;

\node[scale=1.7] (Abrac2) at (15.2/2,0.4){$\}$} ;
\node[scale=0.85] (Amin3) at (15.2/2+0.3,0.4){$\theta$} ;

\node[shape=circle,scale=0.15,fill] (Amin) at (10.77,1){$e_{\min}$} ;
\node[shape=circle,scale=0.15,fill] (Amin2) at (7.17,4.6){$e_{\max}$} ;

\draw[thick,dotted] (15.2/4+0.32,0.8) -- (3*15.2/4,0.8);
%\draw[thick,dotted] (15.2,0.4) -- (3*15.2/4,0.4);
\draw[thick,dotted] (15.2/4+0.32,0.4) -- (0,0.4);

\node[scale=0.85] (Amin3) at (10.77,0.6){$e(t_{\min})$} ;
\node[scale=0.85] (Amin4) at (7.17,4.9){$e(t_{\max})$} ;

\draw[thick,dotted] (7.17-0.8,5) -- (7.17-0.8,3);
\draw[thick,dotted] (7.17+0.8,5) -- (7.17+0.8,3);
\draw[thick,dotted] (7.17-2.1,3) -- (7.17+0.99,3);

\node[scale=0.85] (A0) at (0,-0.49){$0$} ;
\node[scale=0.85] (A1) at (15.2/4+0.1,-0.35){$\frac{1}{4}$} ;
\node[scale=0.85] (A2) at (3*15.2/4-0.2,-0.35){$\frac{3}{4}$} ;
\node[scale=0.85] (A3) at (15.2,-0.49){$1$} ;

\node[scale=0.85] (Abrac2) at (15.2/8,0.2){$\mathbf{\}}$} ;
\node[scale=0.75] (Amin3) at (15.2/8+0.4,0.18){$\theta/2$} ;

\end{tikzpicture}
\end{center}
\caption{\label{fig:BrownianExCap} Visualization of the events $B_1$ and $B_2$ for the Brownian excursion. } 
\end{figure}
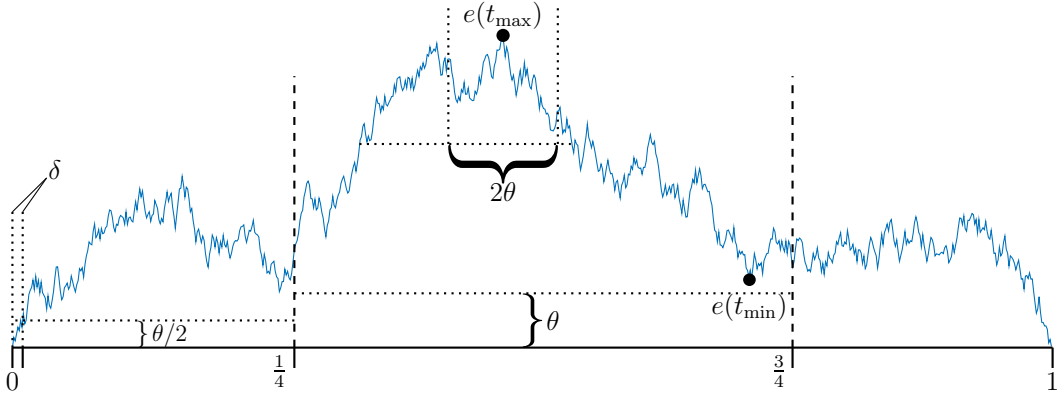
We now use Lemma \ref{lem:ExcursionArea}
to show the following bound on the relaxation time.
\begin{lemma}\label{lem:RelaxationTimeGWT}  For all $\varepsilon>0$, there exists some constant $c_{\varepsilon}>0$ such that for $n$ sufficiently large,
\begin{equation}
P\left( t^n_{\textup{rel}} \geq c_{\varepsilon} n^{\frac{3}{2}} \right) \geq 1 -  \varepsilon\, .
\end{equation} 
\end{lemma}

\begin{proof} Recall the contour function $s$, respectively $|s|$. Fix some $\varepsilon>0$ and let $\theta,\delta>0$ be the constants from Lemma \ref{lem:ExcursionArea}. We start the proof by first choosing a new root $o_{\ast}$ of $G_n$ as follows: Recall the constant $c>0$ from the convergence of $(G_n)$ to the CRT. We set
\begin{equation*}
o_{\ast}:= s(i_{\ast}) \qquad \text{ with } \qquad i_{\ast} = \max\left\{ i \leq \frac{1}{2}n \colon |s(i)| \leq c^{-1}\theta \sqrt{n} \right\} \ .
\end{equation*}
We claim that whenever the event $B_1^n$ given by
\begin{equation*}
B^n_1 := \left\lbrace |s(i)|> c^{-1}\theta  \sqrt{n} \text{ for all } i \in \left[\frac{1}{2} n, \frac{3}{2} n\right]\right\rbrace \cap   \left\lbrace \max_{j \leq  2\delta n} |s(j)|< c^{-1}\frac{\theta}{2}  \sqrt{n} \right\rbrace
\end{equation*} occurs then $o_{\ast}$ is a $\delta$-center of mass. To see this, observe that the subtree $T=T_{o_{\ast}}$, i.e.\ the largest subtree containing the new root $o_{\ast}$ and only sites of distance at least $|o_{\ast}|$ from the old root $o$, must have at least $n/2$ sites. This is due to the fact that the contour traverses, between times $n/2$ and $3n/2$, only edges which belong to $T$ and each edge is visited at most twice. To conclude that $o_{\ast}$ must be a $\delta$-center of mass, note that the tree $\tilde{T}$, induced by the vertices $V \setminus V(T) \cup \{o_{\ast}\}$, contains at least $\delta n$ many sites. This follows from the observation that all edges traversed by the contour until time $2\delta n$ belong to $\tilde{T}$ provided that $B_1^n$ occurs. Thus, $T$ and $\tilde{T}$ satisfy \eqref{def:CenterOfMass} and this proves the claim.\\

Moreover, suppose that in addition the event $B^n_2$ given by 
\begin{equation*}
B^n_2 := \left\lbrace \exists  i \in \left[\left(\frac{1}{2}+2\theta\right) n, \left(\frac{3}{2}-2\theta\right) n\right] \colon  \min_{ j\in [i-2\theta n,i+2 \theta n]} |s(j)| \geq 2c^{-1}\theta \sqrt{n} \right\rbrace 
\end{equation*} occurs. We claim that in this case for $i$ from the event $B_2^n$, the vertex $v$ given by
\begin{equation*}
v:= s(j_{\ast}) \qquad \text{ with } \qquad j_{\ast} = \max\left\{ j < i-2\theta n \colon |s(j)| < 2c^{-1}\theta \sqrt{n} \right\} 
\end{equation*} has at least distance $c^{-1}\theta \sqrt{n}$ from $o_{\ast}$ and it holds that $|T_v| \geq 2\theta n$. Again, this claim follows from the definition of the contour $s$ and the events $B_1^n$ and $B_2^n$ since all edges visited by the contour during $[i-2\theta n,i+2\theta n]$ must belong to $T_v$. \\

Hence, whenever the events $B_1^n$ and $B_2^n$ both occur, Proposition \ref{pro:LowerBoundRelaxationTime} gives a lower bound of $2\delta c^{-1}\theta^2 n^{3/2}$ for the relaxation time. Notice that the events $B_1$ and $B_2$ for the Brownian excursion correspond to the events  $B_1^n$ and $B_2^n$. More precisely, we have that
\begin{align*}
P\left(|s(i)|>c^{-1}\theta \sqrt{n} \text{ for all } i \in \left[ \frac{1}{2}n,\frac{3}{2}n\right] \right) &= P\left(\tilde{s}_n(y)>\theta \text{ for all } y \in \left[ \frac{1}{4},\frac{3}{4}\right] \right) %\\
%&\overset{n\rightarrow \infty}{\longrightarrow} P\left(e(t_{\min})\geq \theta\right)
\end{align*}
converges to $\P\left(e(t_{\min})\geq \theta\right)$ by Lemma \ref{lem:AldousScaling}. A similar statement applies for all other events in the definitions of $B_1^n$ and $B_2^n$. Thus, for all $\varepsilon>0$, Lemma \ref{lem:AldousScaling} yields that there exists an $N=N(\varepsilon)$ such that for all $n\geq N$,
\begin{equation*}
P(B_1^n \cap B_2^n) \geq \P(B_1 \cap B_2) - \frac{\varepsilon}{2}\, .
\end{equation*}
We conclude with Lemma \ref{lem:ExcursionArea}.
\end{proof}

%We now have all tools to show Proposition \ref{pro:NoCutoffGWT2}.

\begin{proof}[Proof of Proposition \ref{pro:NoCutoffGWT2}] Recall from Lemma \ref{lem:ProductCriterion} that a necessary assumption for the simple random walk to satisfy cutoff is that the product criterion \eqref{eq:ProductCriterionIntro} holds. 
In other words, it is enough to show that
\begin{equation}\label{theliminf}
Z_n : = \frac{t^{n}_{\textup{mix}}\left( \frac{1}{4}\right)}{t^n_{\textup{rel}}}  \ \text{ satisfies } \ 
\text P\left( \liminf_{n \rightarrow \infty} Z_n < \infty\right) = 1\, .
\end{equation}
Note that Lemma \ref{lem:MixingTimeGWT} and Lemma \ref{lem:RelaxationTimeGWT} imply (by choosing $K_\varepsilon = \frac{C_{\varepsilon}}{c_{\varepsilon}}$) that 
\begin{equation}\label{is finite}
\forall \varepsilon > 0,  \ \exists K_\varepsilon < \infty \text{ such that }
P\left(Z_n  \leq K_\varepsilon\right) \geq 1-2 \varepsilon\, .
\end{equation}
The following general argument shows that \eqref{is finite} implies \eqref{theliminf}.
Let $\tilde{C} > 0$ and $\tilde{c} \in (0,\tilde{C})$. Then for every $n \in \N$, we have that
\begin{align}\label{eq:Formerindpendence}
P\left(\liminf_{n \rightarrow \infty} Z_n \geq \tilde{C} \right) &\leq P\left( \exists N_0=N_0(\omega) < \infty \colon Z_m \geq \tilde{C} - \tilde{c} \text{ for all } m \geq N_0 \right) \nonumber \\
&\leq P\left(  N_0 \geq n \right) + P\left( Z_n \geq \tilde{C} - \tilde{c}  \right)  \ .
\end{align} 
By choosing $n$ and $\tilde{C} $ large enough, due to \eqref{is finite}, both terms on the right-hand side of
\eqref{eq:Formerindpendence} are $\leq 2\varepsilon$. Since $\varepsilon$ was arbitrary, we conclude that \eqref{theliminf} holds.
\end{proof}

% 
% As a consequence, we obtain that the family of random walks on the \textbf{uniform random tree} which is drawn uniformly at random among the $n^{n-2}$ labeled trees on $n$ vertices, we do not have cutoff. 
% 
% \begin{corollary}\label{cor:NoCutoffUniform} We have that $\P$-almost surely, the family of simple random walks on independent uniform random trees $(G_n)$ on $n$ sites does not exhibit cutoff.
%\end{corollary}
%\begin{proof}
%This follows directly from Proposition \ref{pro:NoCutoffGWT2} and the fact that the uniform random tree with $n$ vertices has the same law as a Galton--Watson tree with a Poisson-$1$ offspring distribution conditioned to have $n$ sites, see \cite{A:CRT2}.
%\end{proof}
%
%\begin{remark} Note that Proposition \ref{pro:NoCutoffGWT2} can be adjusted to cover other families of combinatorial trees which have Aldous' continuum random tree as a scaling limit, and thus have their walks via depth-first search converging to the Brownian excursion.
%\end{remark}

\section{Open questions}

In order to exclude cutoff for the simple random walk on spherically symmetric trees, we assumed that the maximum degree is bounded.

\begin{question} Can the assumption in Proposition \ref{pro:NoCutoffSpherically} on having a bounded maximum degree be relaxed? 
\end{question}

In Section \ref{sec:CutoffAtypical}, we showed that we do not see cutoff for Galton--Watson trees, which are typically \textit{short} and \textit{fat}, see also \cite{AB:ShortAndFat}.

\begin{question} Does a family of trees exist which is \textit{short} and \textit{fat} in the sense of \cite{AB:ShortAndFat} for which the simple random walk exhibits cutoff?
\end{question}

Throughout this article, we used at several points the result of Basu, Hermon and Peres that cutoff occurs for the simple random walk on a family of trees if and only if \eqref{eq:ProductCriterionIntro} is satisfied.

\begin{question} For which families of graphs is the product criterion in \eqref{eq:ProductCriterionIntro} a necessary and sufficient condition such that the simple random walk exhibits cutoff?
\end{question}

\bibliographystyle{plain}

\begin{footnotesize}
\bibliography{Spectrum}
\end{footnotesize}

\textbf{Acknowledgment} We thank Noam Berger and Jonathan Hermon for helpful discussions.
The third author acknowledges the TopMath program and the Studien\-stiftung des deutschen Volkes for financial support.

 \appendix
 
 \section{Proof of the existence of a center of mass} \label{sec:CenterOfMass}
 
 %We now give the proof of Proposition \ref{pro:CenterOfMass}.
 
\begin{proof}[Proof of Proposition \ref{pro:CenterOfMass}] Note that each tree must have a vertex separator $x\in V(G)$, that means that when $x$ is removed from the tree, all of the remaining connected components contain at most $|V(G)|/2$ many sites. This is a well-known fact from graph theory. We will now show that each vertex separator of a tree must also be a $\frac{1}{3}$-center of mass in the above sense. \\

For a given vertex separator $x$, let $S_1,S_2, \dots, S_d$ for some $d\geq 2$ denote the connected components of the tree when  removing $x$. Note that when $d=2$, we take $S_1 \cup \{ x\}$ and $S_2 \cup \{ x\}$ as vertex sets for the trees $T$ and $\tilde{T}$. When $d=3$, associate  $T$ with the largest component $S_i$, which must contain at least $(|V(G)| -1)/3$, but at most $|V(G)|/2$ many sites. Combine the remaining two components for $\tilde{T}$. \\

For $d\geq 4$, we describe now a procedure to reduce the problem of finding the trees $T$ and $\tilde{T}$ for a vertex separator $x$ with $d$ components $(S_i)$, to a problem of finding the trees $T$ and $\tilde{T}$ for $d-1$ components $(\tilde{S}_i)$ which have all size at most $|V(G)|/2$. Hence, solving the problem of finding the trees $T$ and $\tilde{T}$ for $x$ recursively will allow us to conclude the proof. Let $S^{\prime}$ and $S^{\prime\prime}$ be the smallest and second smallest component in $(S_i)$, respectively. By the pigeonhole principle, we have that
\begin{equation}\label{eq:SeparationDecomposition}
 |S^{\prime}| + |S^{\prime\prime}| \leq \left(\frac{1}{4}+\frac{1}{3}\right) (|V(G)|-1) \ .
\end{equation} We distinguish two cases. If the left-hand side in \eqref{eq:SeparationDecomposition} is at least $(|V(G)|-1)/3$, then we associate $T$ with $S^{\prime} \cup S^{\prime\prime}$ and $\tilde{T}$ with its complement, and we are done. If the left-hand side in \eqref{eq:SeparationDecomposition} is at most $(|V(G)|-1)/3$, then remove $S^{\prime}$ and $S^{\prime\prime}$ from $(S_i)$ and replace it by $S^{\prime} \cup S^{\prime\prime}$ to obtain $(\tilde{S}_i)$.
\end{proof}
%\begin{remark} This problem can also be thought as a discrepancy problem on the unit interval: For a given fragmentation of the unit interval into components of sizes at most $\frac{1}{2}$, what is the best way to partition the fragments into two classes, such that the difference of the total sizes is minimized?
%\end{remark}

\end{document}